\documentclass[10pt, reqno]{amsart}
\usepackage{psfrag}
\usepackage[parfill]{parskip}
\usepackage{float, graphicx}
\usepackage[]{epsfig}
\usepackage{amsmath, amsthm, amssymb}
\usepackage{epsfig}
\usepackage{verbatim}
\usepackage{multicol}
\usepackage{url}
\usepackage{latexsym}
\usepackage{mathrsfs}
\usepackage[colorlinks, bookmarks=true]{hyperref}
\usepackage{graphicx}
\usepackage{amsmath}
\usepackage[top=2in, bottom=1.5in, left=1.1in, right=1.1in]{geometry}

\title{Asymptotic Expansion of Spherical Integral}
\date{}
\author{Jiaoyang Huang}
 \address{Department of Mathematics,
Harvard University}
\email{jiaoyang@math.harvard.edu}
\begin{document}
\newtheorem{theorem}{Theorem}
\newtheorem*{theorem*}{Theorem}
\newtheorem{lemma}{Lemma}
\newtheorem{proposition}{Proposition}
\newtheorem{remark}{Remark}
\maketitle
\begin{abstract}
We consider the spherical integral of real symmetric or Hermitian matrices when the rank of one matrix is one. We prove the existence of the full asymptotic expansions of these spherical integrals and derive the first and the second term in the asymptotic expansion. Using asymptotic expression of the spherical integral, we derive the asymptotic freeness of Wigner matrices with (deterministic) Hermitian matrices.
\end{abstract}
\section{Introduction}
In this paper we consider the expansion of the spherical integral
\begin{align}
I_N^{(\beta)}(A_N, B_N)=\int  \exp\{NTr(A_NU^{*}B_NU)\}dm_N^{(\beta)}(U),\label{int1}
\end{align}
where $m_N^{(\beta)}$ is the Haar measure on orthogonal group $O(N)$ if $\beta=1$, on unitary group $U(N)$ if $\beta=2$, and $A_N$, $B_N$ are deterministic $N\times N$ real symmetric or Hermitian matrices, that we can assume diagonal without loss of generality. We follow \cite{GuMa}, to investigate the asymptotics of the spherical integrals under the case $A_{N}=\text{diag}(\theta,0,0,0,0\cdots 0)$:
\begin{align}
\label{eq1}I^{(\beta)}_{N}(A_N, B_N)=I^{(\beta)}_{N}(\theta, B_N)=\int \exp\{\theta N (e_1^{*}B_N e_1)\}dm_N^{(\beta)}(U),
\end{align}
where $e_1$ is the first column of $U$.

The main result of this paper can be stated as follows,
\begin{theorem*}
If $\sup_{N}\|B_N\|_{\infty}<M$, then for any $\theta\in \mathbb{R}$ such that $|\theta|<\frac{1}{4M^2+10M+1}$, the spherical integral has the following asymptotic expansion (up to $O(\frac{1}{N^{n+1}})$ for any given $n$):
\begin{align}
\label{exp1}e^{-\frac{N}{2}\int_0^{2\theta}R_{B_N}(s)ds}I_N(\theta, B_N)
=m_0+\frac{m_1}{N}+\frac{m_2}{N^2}+\cdots +\frac{m_n}{N^n}+O(\frac{1}{N^{n+1}})
\end{align}
where $R_{B_N}$ is the R-transform of the empirical spectral distribute of $B_N$, the coefficients $\{m_i\}_{i=0}^{n}$ depend on $\theta$ and the derivatives of the Hilbert transform of the empirical spectral distribution of $B_N$.
\end{theorem*}

The spherical integral provides a finite dimensional analogue of the $R$-transform in free probability \cite{Ma, Vo}, which states that if $X$ and $Y$ are two freely independent self-adjoint non-commutative random variables, then their $R$-transforms satisfy the following additive formula:
\begin{align*}
R_{X+Y}=R_X+R_Y.
\end{align*}
In the scenario of random matrices, let $\{B_N\}$ and $\{\tilde{B}_N\}$ be sequences of uniformly bounded real symmetric (or Hermitian) matrices whose empirical spectral distributions converge in law towards $\tau_B$ and $\tau_{\tilde{B}}$ respectively. Let $V_{N}$ be a sequence of independent orthogonal (or unitary) matrices following the Haar measure. Then the noncommutative variables $\{B_N\}$ and $\{V_N^*\tilde{B}_NV_N\}$ are asymptotically free. And the law of their sum $\{B_N+V_N^*\tilde{B}_NV_N\}$ converges towards $\tau_{B+V^*\tilde{B}V}$, which is characterized by the following additive formula,
\begin{align}
\label{add1}R_{\tau_{B+V^*\tilde{B}V}}=R_{\tau_{B}}+R_{\tau_{\tilde{B}}}.
\end{align}
We refer to Section 5 in \cite{AGZ} for a proof of this.

For the spherical integral, using the same notation as above, we have the following additive formula
\begin{align*}
\frac{1}{N}\log\mathbb{E}_{V_N}[I_N^{(\beta)}(\theta, B_N+V_N^*\tilde{B}_N V_N)]=\frac{1}{N}\log I_N^{(\beta)}(\theta, B_N)+\frac{1}{N}\log I_N^{(\beta)}(\theta, \tilde{B}_N).
\end{align*}
So if we define
\begin{align*}
R_N(B_N)=\frac{1}{N}\log \mathbb{E}_{B_N}[I_N^{(\beta)}(\theta, B_N)],
\end{align*}
then the additive law can be formulated in a more concise way,
\begin{align}
\label{add2}R_N(B_N+V_N^*\tilde{B}_NV_N)=R_N(B_N)+R_N(\tilde{B}_N).
\end{align}
Compare with (\ref{add1}), which takes advantage of the fact that $\{B_N\}$ and $\{V_N^*\tilde{B}_NV_N\}$ are asymptotically free, and which holds after we take the limit as $N$ goes to infinity. However, the additive formula (\ref{add2}) holds for any $N$, which provides a finite dimensional analogue of the additivity of the $R$-transform. Indeed the $R$-transform is some sort of limit of our $R_N$:
\begin{align*}
\lim_{N\rightarrow\infty}R_N(B_N)=\lim_{N\rightarrow\infty}\frac{1}{N}\log I_{N}^{(\beta)}(\theta,B_N)=\frac{\beta}{2}\int_{0}^{\frac{2\theta}{\beta}}R_{\tau_B}(s)ds.
\end{align*}
Combining this and some concentration of measure inequalities, Guionnet and Ma\"{\i}da showed the additivity of the $R$-transform (\ref{add1}) is a direct consequence of (\ref{add2}) in \cite{GuMa}. Moreover, the finite dimensional analogue of $R$-transform may be used to study the convergence of sum of asymptotically free random matrices. Use the same notations as above, further assume $\{B_N\}$ and $\{\tilde{B}_N\}$ are asymptotically free. From our main results, $R_N(B_N+\tilde{B}_N)$ depends only asymptotically on the empirical measure of $B_N+\tilde{B}_N$. Therefore given the freeness of $\{B_N\}$ and $\{\tilde{B}_N\}$, $R_N(B_N+\tilde{B}_N)$ has the same limit as $R_N(B_N+V_N^*\tilde{B}_NV_N)=R_N(B_N)+R_N(\tilde{B}_N)$. This in turn provides some information about the convergence of the sum $B_N+\tilde{B}_N$.

The asymptotic expansion and some related properties of spherical integrals were thoroughly studied by Guionnet, Ma\"{\i}da and Zeitouni \cite{GuMa, GuZe, GuZeA}. However in their paper, they only studied the first term in the asymptotic expansion. We here derive the higher order asymptotic expansion terms. The proofs are different from those in \cite{GuMa}; Guionnet and Ma\"{\i}da relied on large deviation techniques and used central limit theorem to derive the first order term. In this paper, we express the spherical integral as an integral over only two Gaussian variables, hence allowing for easier asymptotic analysis.

The spherical integral can also be used to study the Schur polynomials. The spherical integral (\ref{int1}) can be expressed in terms of Schur polynomials. The Harish-Chandra-Itzykson-Zuber integral formula \cite{DH, IZ, Me} gives an explicit form for the integral (\ref{int1}) in the case $\beta=2$ and all the eigenvalues of $A$ and $B$ are simple:
\begin{align}\label{add3}
I_N^{(2)}(A, B)= \prod_{i=1}^{N}i! \frac{\det(e^{N\lambda_i(A)\lambda_j(B)})_{1\leq i,j\leq N}}{N^{\frac{N^2-N}{2}}\Delta(\lambda(A))\Delta(\lambda(B))},
\end{align}
where $\Delta$ denotes the Vandermonde determinant,
\begin{align*}
\Delta(\lambda(A))=\prod_{1\leq i<j\leq N}(\lambda_i(A)-\lambda_j(A)),
\end{align*}
and $\lambda_i(\cdot)$ is the $i$-th eigenvalue. If we define the $N$-tuple $\mu=(\lambda_i(B_N)-N+i)_{i=1}^{N}$, then the above expression (\ref{add3}) is the normalized Schur polynomial times an explicit factor,
\begin{align*}
I_N^{(2)}(A,B)
=&\frac{\prod_{i=1}^{N}i!}{N^{\frac{N^2-N}{2}}} \frac{S_{\mu}(e^{N\lambda_1(A)},e^{N\lambda_2(A)},\cdots, e^{N\lambda_N(A)})}{\prod_{i<j}((\mu_i-i)-(\mu_j-j))}\frac{\prod_{ i<j}(e^{N\lambda_i(A)}-e^{N\lambda_j(A)})}{\prod_{i<j }(\lambda_i(A)-\lambda_j(A))}\\
=& \frac{S_{\mu}(e^{N\lambda_1(A)},e^{N\lambda_2(A)},\cdots, e^{N\lambda_N(A)})}{S_{\mu}(1,1,\cdots,1)}\frac{\prod_{ i<j }(e^{N\lambda_i(A)}-e^{N\lambda_j(A)})}{\prod_{i<j}(N\lambda_i(A)-N\lambda_j(A))}.
\end{align*}
In \cite{GoPa}, Gorin and Panova studied the asymptotic expansion of the normalized Schur polynomial $\frac{S_\mu(x_1,x_2,\cdots x_k,1,1,1\dots 1)}{S_\mu(1,1,\cdots,1)}$, for fixed $k$. Its asymptotic expansion can be obtained from a limit formula (Proposition 3.9 in \cite{GoPa}), combining with the asymptotic results for $S_{\mu}(x_i,1,1,\cdots 1)$, which corresponds to the spherical integral where $A$ is of rank one. Therefore our methods give a new proof of some results in \cite{GoPa}.

\textbf{Organization of the Paper.} We summary the notations in Section \ref{sec2}. In Section \ref{sec3}, the main results for orthogonal case are proved. In Section \ref{sec4}, the unitary case are proved, which is an easy corollary of the orthogonal case. In Section \ref{sec5}, we use the spherical integral to prove the asymptotic freeness of Hermitian matrix and Wigner matrix. At last, some technical results are proved in the Appendix.

\section{Some Notations}\label{sec2}
Throughout this paper, $N$ is a parameter going to infinity. We use $O(N^{-l})$ to denote any quantity that is bounded in magnitude by $CN^{-l}$ for some constant $C>0$.

Given a real symmetric or Hermitian matrix $B$, with eigenvalues $\{\lambda_i\}_{i=1}^{N}$. Let $\lambda_{\min}(B)$ ($\lambda_{\max}(B)$) be the minimal (maximal) eigenvalue. We denote the Hilbert transform of its empirical spectral distribution by $H_B$:
\begin{align*}
H_B(z) : \mathbb{R}/[\lambda_{\min}(B),\lambda_{\max}(B)]&\mapsto \mathbb{R}\\
                                          z &\mapsto \frac{1}{N}\sum_{i=1}^{N}\frac{1}{z-\lambda_i}.
\end{align*}
On intervals $(-\infty,\lambda_{\min}(B))$ and $(\lambda_{\max}(B),+\infty)$, $H_B$ is monotonous. The $R$-transform of empirical spectral distribution of $B$ is
\begin{align*}
R_B(z)=H_B^{-1}(z)-\frac{1}{z},
\end{align*}
on $\mathbb{R}/\{0\}$, where $H_B^{-1}$ is the functional inverse of $H_B$ from $\mathbb{R}/\{0\}$ to $(-\infty,\lambda_{\min}(B))\cup(\lambda_{\max}(B),+\infty)$. In the following paper, for given non-zero real number $\theta$, we denote $v(\theta)=R_B(2\theta)$, for simplicity we will omit the variable $(\theta)$ in the expression of $v(\theta)$.

We denote the normalized $k$-th derivative of Hilbert transform
\begin{align*}
A_k=\frac{(-1)^{k-1}}{(k-1)!(2\theta)^{k}}\frac{d^{k-1}H_B}{dz^{k-1}}(v+\frac{1}{2\theta})
=\frac{1}{N}\sum_{i=1}^{N}\frac{1}{(1-2\theta\lambda_i+2\theta v)^k}.
\end{align*}
Notice from the definition of $v$, $A_1=1$. The coefficients in the asymptotic expansion (\ref{exp1}), can be represented in terms of these $A_k$'s.

Also we denote
\begin{align*}
F=&\frac{1}{N}\sum_{i=1}^{N}\frac{\lambda_i}{(1-2\theta\lambda_i+2\theta v)^2}=-\frac{1}{2\theta}+\frac{1+2\theta v}{2\theta}A_2,\\
G=&\frac{1}{N}\sum_{i=1}^{N}\frac{\lambda_i^2}{(1-2\theta\lambda_i+2\theta v)^2}=-\frac{1+4\theta v}{4\theta^2}+\frac{(1+2\theta v)^2}{4\theta^2}A_2.
\end{align*}

\section{Orthogonal Case}{\label{sec3}}

\subsection{First Order Expansion}
In this section we consider the real case, $\beta=1$. For simplicity, in this section we will omit the superscript ($\beta$) in all the notations. We can assume $B_N$ is diagonal $B_{N}=\text{diag}(\lambda_1,\lambda_2,\cdots, \lambda_N)$. Notice $e_1$ is the first column of $U$, which follows Haar measure on orthogonal group $O(N)$. $e_1$ is uniformly distributed on $S^{N}$ (orthogonally invariant) and can be represented as the normalized Gaussian vector,
\begin{align*}
e_1=\frac{g}{\|g\|},
\end{align*}
where $g=(g_1,g_2,\cdots, g_N)^T$ is the standard Gaussian vector in $\mathbb{R}^N$, and $\|\cdot\|$ is the Euclidean norm in $\mathbb{R}^N$.
Plug them into (\ref{eq1})
\begin{align}
\label{eq2}I_N(\theta, B_N)=\int \exp\left\{ \theta N \frac{\lambda_1g_1^2+\lambda_2g_2^2\cdots +\lambda_Ng_N^2}{g_1^2+g_2^2+\cdots g_N^2}\right\} \prod_{i=1}^{N}dP(g_i).
\end{align}
where $P(\cdot)$ is the standard Gaussian probability measure on $\mathbb{R}$. Following the paper \cite{GuMa}, define
\begin{align*}
\gamma_N=\frac{1}{N}\sum_{i=1}^{N}g_i^2-1,\qquad
\hat{\gamma}_N=\frac{1}{N}\sum_{i=1}^{N}\lambda_ig_i^2-v.
\end{align*}

Then (\ref{eq2}) can be rewritten in the following form
\begin{align*}
I_N(\theta, B_N)=\int\exp\{\theta N \frac{\hat{\gamma}_N+v}{\gamma_N+1}\}\prod_{i=1}^{N}dP(g_i).
\end{align*}
\begin{remark}
Notice $v$ is defined in Section \ref{sec2}, and satisfies the following explicit formulas:
\begin{align*}
\frac{1}{N}\sum_{i=1}^{N}\frac{1}{1-2\theta\lambda_i+2\theta v}=1,\qquad \frac{1}{N}\sum_{i=1}^{N}\frac{\lambda_i}{1-2\theta\lambda_i+2\theta v}=v.
\end{align*}
Next we will do the following change of measure
\begin{align*}
P_N(dg_1,dg_2,\cdots,dg_N)=\frac{1}{\sqrt{2\pi}^N}\prod_{i=1}^{N}\left(\sqrt{1+2\theta v-2\theta\lambda_i}e^{-\frac{1}{2}(1+2\theta v-2\theta\lambda_i)g_i^2}dg_i\right)
\end{align*}
With this measure, we have $\mathbb{E}(\gamma_N)=0$ and $\mathbb{E}(\hat{\gamma}_N)=0$. Therefore intuitively from the law of large number, as $N$ goes to infinity, $\gamma_N$ and $\hat{\gamma}_N$ concentrate at origin. This is exactly what we will do in the following.
\end{remark}

Let us also define
\begin{align*}
I_N^{\kappa_1,\kappa_2}(\theta, B_N)=\int_{|\gamma_N|\leq N^{-\kappa_1},\atop |\hat{\gamma}_N|\leq N^{-\kappa_2}}\exp\{\theta N \frac{\hat{\gamma}_N+v}{\gamma_N+1}\}\prod_{i=1}^{N}dP(g_i),
\end{align*}
where constants $\kappa_1$ and $\kappa_2$ satisfy $\frac{1}{2}>\kappa_1>2\kappa_2$ and $2\kappa_1+\kappa_2>1$. We prove the following proposition, then the difference between $I_N(\theta, B_N)$ and $I_N^{\kappa_1,\kappa_2}(\theta, B_N)$ is of order $O(N^{-\infty})$. A weaker form of this proposition also appears in \cite{GuMa}. With this proposition, for the asymptotic expansion, we only need to consider $I_N^{\kappa_1,\kappa_2}(\theta, B_N)$.
\begin{proposition}{\label{proposition1}}
Given constants $\kappa_1$ and $\kappa_2$ satisfying $\frac{1}{2}>\kappa_1>2\kappa_2$ and $2\kappa_1+\kappa_2>1$. There exist constants $c$, $c'$, depending on $\kappa_1$,$\kappa_2$ and $\sup \|B_N\|_{\infty}$, such that for $N$ large enough
\begin{align}
\label{eq3}|I_N(\theta, B_N)-I_N^{\kappa_1,\kappa_2}(\theta, B_N)|\leq c e^{-c'N^{1-2\kappa_1}}I_N(\theta, B_N).
\end{align}
\end{proposition}

Before proving this proposition, we state the following useful lemma of concentration measure on sum of squares of Gaussian random variables.
\begin{lemma}\label{CMI}
Given independent gaussian random variables $\{g_i\}_{i=1}^{N}$, consider the weighted sum $\sum_{i=1}^{N}{a_i g_i^2}$. If there exists some constant $c>0$, such that $\max\{|a_i|\}\leq \frac{c}{\sqrt{N}}$, then for $N$ large enough, the weighted sum satisfies the following concentration inequality,
\begin{align*}
\mathbb{P}(|\sum_{i=1}^{N}a_i(g_i^2-1)|\geq N^{\kappa})\leq 2e^{-c' N^{2\kappa}},  \qquad 0<\kappa<\frac{1}{2}.
\end{align*}
\end{lemma}
\begin{proof}
This can be proved by applying Markov's inequality to $\exp\{t\sum_{i=1}^{N}a_i(g_i^2-1)\}$ for some well-chosen value of $t$.
\end{proof}

\begin{proof}
We can split (\ref{eq3}) into two parts
\begin{align*}
I_N(\theta, B_N)-I_N^{\kappa_1,\kappa_2}(\theta, B_N)
=\underbrace{\left(\int -\int_{|\gamma_N|\leq N^{-\kappa_1}}\right)}_{E_1}
+\underbrace{\left(\int_{|\gamma_N|\leq N^{-\kappa_1}} -\int_{|\gamma_N|\leq N^{-\kappa_1},\atop |\hat{\gamma_N}|\leq N^{-\kappa_2}}\right)}_{E_2}
\end{align*}
Follow the same argument as in Lemma $14$(\cite{GuMa}),
\begin{align*}
|E_1|\leq ce^{-c'N^{1-2\kappa_1}}I_{N}(\theta, B_N).
\end{align*}
For $E_2$, notice that
\begin{align*}
&\exp\{\theta N \frac{\hat{\gamma}_N+v}{\gamma_N+1}\}\prod_{i=1}^{N}dP(g_i)\\
=& \exp\left\{\theta Nv+\theta N(\hat{\gamma}_N-v\gamma_N)- \theta N \gamma_N\frac{\hat{\gamma}_N-v\gamma_N}{\gamma_N+1}\right\}\prod_{i=1}^{N}dP(g_i)\\
=&e^{\theta N v} \exp\left\{-\theta N \gamma_N\frac{\hat{\gamma}_N-v\gamma_N}{\gamma_N+1}\right\}\prod_{i=1}^{N}\frac{1}{\sqrt{2\pi}} e^{-\frac{1}{2}g_i^2(1-2\theta\lambda_i+2\theta v)}dg_i
\end{align*}
With the change of variable $g_i=\frac{\tilde{g}_i}{\sqrt{1-2\theta\lambda_i+2\theta v}}$,
\begin{align*}
&\exp\{\theta N \frac{\hat{\gamma}_N+v}{\gamma_N+1}\}\prod_{i=1}^{N}dP(g_i)\\
=&e^{\theta N v-\frac{1}{2}\sum_{i=1}^{N}\log(1-2\theta\lambda_i+2\theta v)} \exp\left\{-\theta N \gamma_N\frac{\hat{\gamma}_N-v\gamma_N}{\gamma_N+1}\right\}\prod_{i=1}^{N}dP(\tilde{g}_i).
\end{align*}
For the exponential term $\theta v-\frac{1}{2N}\sum_{i=1}^{N}\log(1-2\theta\lambda_i+2\theta v)$, if we take derivative with respect to $\theta$, and then integrate from $0$, we get that  
\begin{align*}
\theta v-\frac{1}{2N}\sum_{i=1}^{N}\log(1-2\theta\lambda_i+2\theta v)=\frac{1}{2}\int_{0}^{2\theta}R_{B_N}(s)ds.
\end{align*}
We can bound $|E_2|$,
\begin{align*}
|E_2|=&\left|\int_{|\gamma_N|\leq N^{-\kappa_1},\atop |\hat{\gamma}_N|>N^{-\kappa_2}}\exp\{\theta N \frac{\hat{\gamma}_N+v}{\gamma_N+1}\}\prod_{i=1}^{N}dP(g_i)\right|\\
=&e^{\frac{N}{2}\int_0^{2\theta}R_{B_N}(s)ds}\left|\int_{|\gamma_N|\leq N^{-\kappa_1},\atop|\hat{\gamma}_N|>N^{-\kappa_2}}\exp\{\theta N \gamma_N\frac{\hat{\gamma}_N-v\gamma_N}{\gamma_N+1}\}\prod_{i=1}^{N}
dP(\tilde{g}_i)\right|\\
\leq &e^{\frac{N}{2}\int_0^{2\theta}R_{B_N}(s)ds}e^{|\theta| (\|B_N\|_{\infty}+|v|)N^{1-\kappa_1}}P(|\hat{\gamma}_N|> N^{-\kappa_2}),
\end{align*}
where
\begin{align*}
P(|\hat{\gamma}_N|>N^{-\kappa_2})
=&P(|\frac{1}{N}\sum_{i=1}^{N}\frac{\lambda_i\tilde{g}_i^2}{1-2\theta\lambda_i+2\theta v}-v|>N^{-\kappa_2})\\
=&P(|\frac{1}{\sqrt{N}}\sum_{i=1}^{N}\frac{\lambda_i(\tilde{g}_i^2-1)}{1-2\theta\lambda_i+2\theta v}|>N^{\frac{1}{2}-\kappa_2})\\
\leq& 2e^{-c'N^{1-2\kappa_2}}.
\end{align*}
The last inequality comes from the concentration measure inequality in Lemma \ref{CMI} and the uniformly boundedness of $\|B_N\|_{\infty}$.
Moreover from Lemma $14$ of \cite{GuMa}, we have the lower bound for $I_N(\theta, B_N)$,
\begin{align*}
I_N(\theta,B_N)\geq
&ce^{\frac{N}{2}\int_0^{2\theta}R_{B_N}(s)ds}e^{-|\theta| (\|B_N\|_{\infty}+v)N^{1-\kappa_1}}.
\end{align*}
From our assumption $\kappa_1>2\kappa_2$, so $e^{|\theta| (\|B_N\|_{\infty}+v)N^{1-\kappa_1}}\ll e^{c'N^{1-2\kappa_2}}$. Therefore for $N$ large enough
\begin{align*}
|E_2|\leq ce^{-c'N^{1-2\kappa_2}}I_N(\theta,B_N).
\end{align*}
This finishes the proof of (\ref{eq3}).
\end{proof}

Since the difference between $I_N(\theta, B_N)$ and $I_N^{\kappa_1,\kappa_2}(\theta, B_N)$ is of order $O(N^{-\infty})$, for asymptotic expansion, we only need to consider $I_N^{\kappa_1,\kappa_2}(\theta, B_N)$.
\begin{align*}
I^{\kappa_1,\kappa_2}_N(\theta, B_N)
= e^{\frac{N}{2}\int_0^{2\theta}R_{B_N}(s)ds}\int_{|\gamma_N|\leq N^{-\kappa_1},\atop |\hat{\gamma}_N|\leq N^{-\kappa_2}}\exp\{-\theta N\gamma_N\frac{\hat{\gamma}_N-v\gamma_N}{\gamma_N+1}\}
\prod_{i=1}^{N}dP(\tilde{g}_i),
\end{align*}

The next thing is to expand the following integral,
\begin{align}
\label{eq4}\int_{|\gamma_N|\leq N^{-\kappa_1},\atop |\hat{\gamma}_N|\leq N^{-\kappa_2}}\exp\{-\theta N\gamma_N\frac{\hat{\gamma}_N-v\gamma_N}{\gamma_N+1}\}
\prod_{i=1}^{N}dP(\tilde{g}_i).
\end{align}
Since the denominator $\gamma_N+1$ in the exponent has been restricted in a narrow interval centered at $1$, we can somehow "ignore" it by Taylor expansion, which results in the following error $R_1$,
\begin{align}
\notag R_1=&\Big{|}\int_{|\gamma_N|\leq N^{-\kappa_1},\atop |\hat{\gamma}_N|\leq N^{-\kappa_2}}\exp\{-\theta N\gamma_N\frac{\hat{\gamma}_N-v\gamma_N}{\gamma_N+1}\}
\prod_{i=1}^{N}dP_i(g_i)\\
\notag -&\int_{|\gamma_N|\leq N^{-\kappa_1},\atop |\hat{\gamma}_N|\leq N^{-\kappa_2}}\exp\{-\theta N \gamma_N(\hat{\gamma}_N-v\gamma_N)\}\prod_{i=1}^{N}dP_i(g_i)\Big{|}\\
\label{er1}=&\left|\int_{|\gamma_N|\leq N^{-\kappa_1}, \atop |\hat{\gamma}_N|\leq N^{-\kappa_2}}\exp\{-\theta N \gamma_N(\hat{\gamma}_N-v\gamma_N)\}(\exp\{-\theta N\gamma_N^2\frac{\hat{\gamma}_N-v\gamma_N}{\gamma_N+1}\}-1)
\prod_{i=1}^{N}dP_i(g_i)\right|.
\end{align}
Under our assumption $2\kappa_1+\kappa_2>1$, $-\theta N\gamma_N^2\frac{\hat{\gamma}_N-v\gamma_N}{\gamma_N+1}=O(N^{1-2\kappa_1-\kappa_2})=o(1)$, which means the above error $R_1$ is of magnitude $o(1)$. Therefore the error $R_1$ won't contribute to the first term in asymptotic expansion of integral (\ref{eq4}). And the first term of asymtotics of the spherical integral comes from the following integral
\begin{align}
\label{firstt}\int_{|\gamma_N|\leq N^{-\kappa_1},\atop |\hat{\gamma}_N|\leq N^{-\kappa_2}}\exp\{-\theta N \gamma_N(\hat{\gamma}_N-v\gamma_N)\}\prod_{i=1}^{N}dP_i(g_i).
\end{align}

We will revisit this for higher order expansion in next section. In the rest part of this section, we prove the following theorem about the first term of asymptotic expansion of spherical integral. This strengthens Theorem $3$ in \cite{GuMa}, where they require additional conditions on the convergence speed of spectral measure).

\begin{theorem}\label{rf}
If $\sup_{N}\|B_N\|_{\infty}<M$, then for any $\theta\in \mathbb{R}$ such that $|\theta|<\frac{1}{4M^2+10M+1}$, the spherical integral has the following asymptotic expansion,
\begin{align}
\notag&e^{-\frac{N}{2}\int_0^{2\theta} R_{B_N}(s)ds}I_N(\theta, B_N)\\
\label{eq2.1}=&\int_{|\gamma_N|\leq N^{-\kappa_1},\atop |\hat{\gamma}_N|\leq N^{-\kappa_2}}\exp\{-\theta N \gamma_N(\hat{\gamma}_N-v\gamma_N)\}\prod_{i=1}^{N}dP_i(g_i)+o(1)=\frac{1}{\sqrt{A_2}}+o(1),
\end{align}
where $R_{B_N}$ is the R-transform of the empirical spectral measure of $B_N$, and $A_2=\frac{1}{N}\sum_{i=1}^{N}\frac{1}{(1+2\theta v-2\theta \lambda_i)^2}$.
\end{theorem}

Notice (\ref{eq2.1}) is an integral of $N$ Gaussian variables and the exponent $-\theta N \gamma_N(\hat{\gamma}_N-v\gamma_N)$ is a quartic polynomial in terms of $g_i$'s; So one can not compute the above integral directly. Using the following lemma, by introducing two more Gaussian variables $x_1$ and $x_2$, we can reduce the degree of the exponent to two, and  it turns to be ordinary Gaussian integral. Then we can directly compute it, and obtain the higher order asymptotic expansion.
\begin{lemma}
For any $\alpha\in \mathbb{C}$ with $|\alpha|\leq CN^{\kappa}$,
\begin{align*}
e^{\frac{\alpha^2}{2}}=\frac{1}{\sqrt{2\pi}}\int_{I} e^{-\frac{x^2}{2}}e^{-x\alpha}dx + O(N^{-\infty}),
\end{align*}
where interval $I=[-N^{\kappa+\epsilon}, N^{\kappa+\epsilon}]$ for any $\epsilon>0$.
\begin{proof}
Recall the formula, for any $\alpha\in \mathbb{C}$
\begin{align*}
\frac{1}{\sqrt{2\pi}}\int_{\mathbb{R}} e^{-\frac{(x+\alpha)^2}{2}}=1.
\end{align*}
The main contribution of the above integral comes from where $|x+\alpha|$ is small. More precisely,
\begin{align*}
&\left|e^{\frac{\alpha^2}{2}}\int_{\mathbb{R}} e^{-\frac{(x+\alpha)^2}{2}}dx-e^{\frac{\alpha^2}{2}}\int_{I} e^{-\frac{(x+\alpha)^2}{2}}\right|
\leq  e^{\frac{|\alpha|^2}{2}}\left|\int_{I^c}e^{-\frac{(x+\alpha)^2}{2}}dx
\right|\\
\leq &2e^{\frac{|\alpha|^2}{2}}
\int_{N^{\kappa+\epsilon}}^{\infty}e^{-\frac{(x-|\alpha|)^2}{2}}dx
\leq
2e^{\frac{C^2N^{2\kappa}}{2}}\int_{N^{\kappa+\epsilon}-CN^{\kappa}}^{\infty}e^{-\frac{x^2}{2}}dx
\leq
c' e^{-cN^{2\kappa+2\epsilon}}=O(N^{-\infty}),
\end{align*}
where $c$ and $c'$ are constants independent of $N$. Therefore
\begin{align*}
e^{\frac{\alpha^2}{2}}
=e^{\frac{\alpha^2}{2}}\frac{1}{\sqrt{2\pi}}\int e^{-\frac{(x+\alpha)^2}{2}}=\frac{1}{\sqrt{2\pi}}\int_{I} e^{-\frac{x^2}{2}}e^{-x\alpha}dx + O(N^{-\infty}).
\end{align*}
\end{proof}

\end{lemma}
Write the exponent in integral (\ref{firstt}) as sum of two squares, then we can implement the above lemma, to reduce its degree to two. Let $b^2=\frac{\theta}{2}$, $ I_1=[-N^{\frac{1}{2}-\kappa_1+\epsilon},N^{\frac{1}{2}-\kappa_1+\epsilon}]$ and $ I_2=[-N^{\frac{1}{2}-\kappa_2+\epsilon},N^{\frac{1}{2}-\kappa_2+\epsilon}]$, for some $\epsilon>0$.  Then on the region $\{|\gamma_N|\leq N^{-\kappa_1}, |\hat{\gamma}_N|\leq N^{-\kappa_2}\}$
\begin{align*}
&\exp\{-\theta N \gamma_N(\hat{\gamma}_N-v\gamma_N)\}\\
=&\exp\{-\frac{\theta}{2} \frac{(\sqrt{N}((1-v)\gamma_N+\hat{\gamma}_N))^2-(\sqrt{N}((1+v)\gamma_N-\hat{\gamma}_N))^2}{2}
\}\\
=&\frac{1}{(\sqrt{2\pi})^2}\int_{I_2}\int_{I_1} e^{-ibx_1\sqrt{N} ((1-v)\gamma_N+\hat{\gamma}_N)}e^{-bx_2\sqrt{N} ((1+v)\gamma_N-\hat{\gamma}_N)}
e^{-\frac{x_1^2}{2}}e^{-\frac{x_2^2}{2}}dx_1dx_2+O(N^{-\infty}).
\end{align*}
Plug it back into (\ref{firstt}), ignoring the $O(N^{-\infty})$ error,
\begin{align}
\notag&\int_{|\gamma_N|\leq N^{-\kappa_1},\atop |\hat{\gamma}_N|\leq N^{-\kappa_2}}\exp\{-\theta N\gamma_N(\hat{\gamma}_N-v\gamma_N)\}
\prod_{i=1}^{N}dP(\tilde{g}_i)\\
\notag=&\int_{|\gamma_N|\leq N^{-\kappa_1},\atop |\hat{\gamma}_N|\leq N^{-\kappa_2}} \frac{1}{(\sqrt{2\pi})^2}\left\{\int_{I_2}\int_{I_1} e^{-ibx_1\sqrt{N} ((1-v)\gamma_N+\hat{\gamma}_N)}e^{-bx_2\sqrt{N} ((1+v)\gamma_N-\hat{\gamma}_N)}
e^{-\frac{x_1^2}{2}}e^{-\frac{x_2^2}{2}}dx_1dx_2\right\}\prod_{i=1}^{N} dP(\tilde{g}_i)\\
\label{eq5}=&\frac{1}{(\sqrt{2\pi})^2}\int_{I_1\times I_2}e^{ibx_1\sqrt{N}}e^{bx_2\sqrt{N}} \left\{\int_{|{\gamma}_N|\leq N^{-\kappa_1}\atop |\hat{\gamma}_N|\leq N^{-\kappa_2}}\prod_{i=1}^{N}\frac{1}{\sqrt{2\pi}}e^{-\frac{\tilde{g}_i^2}{2}
\left(1+\frac{2b(i(1-v+\lambda_i) x_1+(1+v-\lambda_i)x_2)}{\sqrt{N}(1-2\theta \lambda_i+2\theta v)}\right)}\prod_{i=1}^{N}d\tilde{g}_i\right\}\prod_{i=1}^{2}e^{-\frac{x_i^2}{2}}dx_i.
\end{align}

Notice here in the inner integral, the integral domain is the region $\mathcal{D}=\{|\gamma_N|\leq N^{-\kappa_1}, |\hat{\gamma}_N|\leq N^{-\kappa_2}\}$ and the Gaussian variables $\tilde{g}_i$ are located in this region with overwhelming probability. It is highly likely that if we instead integrate over the whole space $\mathbb{R}^{N}$, the error is exponentially small. We will first compute the integral under the belief that the integral outside this region $\mathcal{D}$ is negligible, then come back to this point later. Replace the integral region $\mathcal{D}$ by $\mathbb{R}^N$,
\begin{align}
\label{pr4}&e^{ibx_1\sqrt{N}}e^{bx_2\sqrt{N}} \left\{\int_{\mathbb{R}^N}\prod_{i=1}^{N}\frac{1}{\sqrt{2\pi}}e^{-\frac{\tilde{g}_i^2}{2}\left(1+\frac{2b(i(1-v+\lambda_i) x_1+(1+v-\lambda_i)x_2)}{\sqrt{N}(1-2\theta \lambda_i+2\theta v)}\right)}\prod_{i=1}^{N}d\tilde{g}_i\right\}\\
\notag=&\prod_{i=1}^{N} \frac{\exp\{\frac{b(i(1-v+\lambda_i) x_1+(1+v-\lambda_i)x_2)}{\sqrt{N}(1-2\theta \lambda_i+2\theta v)}\}}{\sqrt{1+\frac{2b(i(1-v+\lambda_i) x_1+(1+v-\lambda_i)x_2)}{\sqrt{N}(1-2\theta \lambda_i+2\theta v)}}}.
\end{align}

Let $\mu_i=\frac{b(i(1-v+\lambda_i) x_1+(1+v-\lambda_i)x_2)}{\sqrt{N}(1-2\theta \lambda_i+2\theta v)}$. Then $\mu_i=O(N^{\epsilon-\kappa_2})$, where $0<\epsilon<\kappa_2$. So we have the Taylor expansion
\begin{align*}
\prod_{i=1}^{N}\frac{e^{\mu_i}}{\sqrt{1+2\mu_i
}}
=\prod_{i=1}^{N}e^{\mu_i
^2+o(\mu_i^2)}
=\prod_{i=1}^{N}e^{\mu_i
^2}(1+o(\sum_{i=1}^{N}\mu_i^2)),
\end{align*}
later we will see $\int e^{\sum_{i=1}^{N}\mu^2}\sum_{i=1}^{N}\mu_i^2dP(x_1)dP(x_2)$=O(1). Thus, the first term in the asymptotics of integral (\ref{eq5}) comes from the integral of $e^{\sum_{i=1}^{N}\mu_i^2}$, which is
\begin{align}
\label{pre5}\frac{1}{(\sqrt{2\pi})^2}\int_{I_2}\int_{I_1}\exp\left\{\frac{\theta}{2v} \sum_{i=1}^{N}\frac{\left(i(1-v+\lambda_i)x_1+(1+v-\lambda_i)x_2\right)^2}{(1-2\theta\lambda_i+2\theta v)^2}-\frac{x_1^2+x_2^2}{2}\right\}dx_1dx_2.
\end{align}
The exponent is a quadratic form, and can be written as $-\frac{1}{2}\langle x, Kx\rangle$ where $K$ is the following $2\times 2$ symmetric matrices,
\begin{align*}
K=
 \left[ \begin{array}{cc}
1+\theta\left((1-v)^2A_2+2(1-v)F+G\right) & -\theta i\left((1-v^2)A_2+2vF-G\right) \\
-\theta i\left((1-v^2)A_2+2vF-G\right) & 1-\theta\left((1+v)^2A_2-2(1+v)F+G\right)  \end{array} \right]
\end{align*}
where $A_2$, $F$, $G$ are respectively $\frac{1}{N}\sum_{i=1}^{N}\frac{1}{(1-2\theta \lambda_i+2\theta v)^2}$, $\frac{1}{N}\sum_{i=1}^{N}\frac{\lambda_i}{(1-2\theta \lambda_i+2\theta v)^2}$ and $\frac{1}{N}\sum_{i=1}^{N}\frac{\lambda_i^2}{(1-2\theta \lambda_i+2\theta v)^2}$. With this notation, the above integral (\ref{pre5}) can be rewritten as
\begin{align}
\label{pre6}\frac{1}{(\sqrt{2\pi})^2}\int_{I_2}\int_{I_1} e^{-\frac{1}{2}\langle x, Kx\rangle}dx_1dx_2.
\end{align}

To deal with this complex Gaussian integral, we have the following lemma
\begin{lemma}\label{CGI}
If an $n$ by $n$ symmetric matrix $K$ can be written as $K=A+iB$, where $A$ is a real positive definite matrix, $B$ is a real symmetric matrix. Then we have the Gaussian integral formula
\begin{align*}
\frac{1}{(\sqrt{2\pi})^n}\int_{\mathbb{R}^n} e^{-\frac{1}{2}\langle x, Kx\rangle}dx=\frac{1}{\sqrt{det(K)}}.
\end{align*}
\begin{proof}
Since $A$ is positive definite, there exists a positive definite matrix $C$ such that $A=C^TC$. Since $CBC^{T}$ is symmetric, it can be diagonalized by some special orthogonal matrix $U$, let $P=UC$. Then we have $A=P^TP$ and $B=P^T D P$, where $D=diag\{d_1,d_2,\cdots, d_n\}$. Thus $K=A+iB=P^T(I+iD)P$. Plug this back to the integral
\begin{align}
\notag\frac{1}{(\sqrt{2\pi})^n}\int_{\mathbb{R}^N} e^{-\frac{1}{2}\langle x, Kx\rangle}dx=&\frac{1}{(\sqrt{2\pi})^n}\int_{\mathbb{R}^N} e^{-\frac{1}{2}\langle x, P^T(I+iD)P x\rangle}dx\\
\notag=&\frac{1}{(\sqrt{2\pi})^n \det(P)}\int_{\mathbb{R}^N} e^{-\frac{1}{2}\langle y, (I+iD)y\rangle}dy
=\frac{1}{(\sqrt{2\pi})^n\det(P)}\prod_{k=1}^{n}\int_{\mathbb{R}^N} e^{-\frac{(1+id_k)}{2}y_k^2}dy_k\\
\label{eq6}=&\frac{1}{(\sqrt{2\pi})^n\det(P)}\prod_{k=1}^{n}\sqrt{\frac{2\pi}{1+id_k}}=\frac{1}{\sqrt{\det(K)}},
\end{align}
where the square root in (\ref{eq6}) is the branch with positive real part, in our case $K$ is a $2\times 2$ matrix.
\end{proof}
\end{lemma}
Now we need to verify that our matrix $K$ satisfies the condition of Lemma \ref{CGI}. Write $K$ as the following sum
\begin{align*}
K&=
 \left[ \begin{array}{cc}
1+\theta\left((1-v)^2A_2+2(1-v)F+G\right) & 0 \\
0 & 1-\theta\left((1+v)^2A_2-2(1+v)F+G\right)  \end{array} \right]
\\&+
 i\left[ \begin{array}{cc}
0 & -\theta \left((1-v^2)A_2+2vF-G\right) \\
-\theta \left((1-v^2)A_2+2vF-G\right) & 0  \end{array} \right]
\end{align*}
Since $\min\{\lambda_i\}\leq v(\theta) \leq \max\{\lambda_i\}$, $|v(\theta)|<\max|\lambda_i|<M$. If $\theta<\frac{1}{4M^2+10M+1}$, it is easy to check the real part of matrix $K$ is positive definite. To use lemma \ref{CGI}, the only thing is to compute the determinant of matrix $K$. Notice the algebraic relations between parameters $A_2$, $F$, $G$, $\det(K)=A_2$. Therefore we obtain the first order asymptotic of the spherical integral,
\begin{align*}
e^{-\frac{N}{2}\int_0^{2\theta} R_{B_N}(s)ds}I_N(\theta, B_N)
=\frac{1}{(\sqrt{2\pi})^2}\int_{I_2}\int_{I_1} e^{-\frac{1}{2}\langle x , Kx\rangle}(1+o(1))dx=\frac{1}{\sqrt{A_2}}+o(1).
\end{align*}

Go back to integral (\ref{eq5}), we need to prove that the integral outside the region $\mathcal{D}$ is of order $O(N^{-\infty})$, then replacing the integral domain $\mathcal{D}=\{|\hat{\gamma}_N|\leq N^{-\kappa_1}, |\gamma_N|\leq N^{-\kappa_2}\}$ by the whole space $\mathbb{R}^N$ won't affect the asymptotic expansion.
\begin{lemma}
Consider the integral (\ref{eq5}) on the complement of $\mathcal{D}$, i.e. on $\{|\hat{\gamma}_N|\geq N^{-\kappa_1}\text{ or  }|\gamma_N|\geq N^{-\kappa_2}\}$,
\begin{align*}
R=&\Big{|}\int_{I_2}\int_{I_1}e^{ibx_1\sqrt{N}+bx_2\sqrt{N}} \{\int_{|\hat{\gamma}_N|\geq N^{-\kappa_1}\atop \text{or }|\gamma_N|\geq N^{-\kappa_2}}\prod_{i=1}^{N}\frac{1}{\sqrt{2\pi}}e^{-\frac{\tilde{g}_i^2}{2}\left(1+\frac{2b(i(1-v+\lambda_i) x_1+(1+v-\lambda_i)x_2)}{\sqrt{N}(1-2\theta \lambda_i+2\theta v)}\right)}\prod_{i=1}^{N}d\tilde{g}_i\}e^{-\frac{x_1^2}{2}}e^{-\frac{x_2^2}{2}}dx_1dx_2\Big{|}.
\end{align*}
Then for $N$ large enough, the error $R\leq c'e^{-cN^{\zeta}}$, for some constant $c,c',\zeta>0$ depending on $\kappa_1$, $\kappa_2$ and $\|B_N\|_{\infty}$.
\end{lemma}
\begin{proof}
\begin{align*}
R\leq \int_{I_2}\int_{I_1}e^{\Re\{ibx_1\sqrt{N}+bx_2\sqrt{N}\}} \{
\int_{|\hat{\gamma}_N|\geq N^{-\kappa_1}\atop \text{or }|\gamma_N|\geq N^{-\kappa_2}}\prod_{i=1}^{N}\frac{1}{\sqrt{2\pi}}e^{-\Re\left\{\frac{\tilde{g}_i^2}{2}\left(1+\frac{2b(i(1-v+\lambda_i) x_1+(1+v-\lambda_i)x_2)}{\sqrt{N}(1-2\theta \lambda_i+2\theta v)}\right)\right\}}\prod_{i=1}^{N}d\tilde{g}_i\}e^{-\frac{x_1^2}{2}}e^{-\frac{x_2^2}{2}}dx_1dx_2.
\end{align*}
Since $b=\sqrt{\frac{\theta}{2}}$, either $b$ is real or imaginary. For simplicity here we only discuss the case when $b$ is real. The case when $b$ is imaginary can be proved in the same way. The above integral can be simplified as
\begin{align*}
R\leq \int_{I_2}e^{bx_2\sqrt{N}} \left\{\int_{|\hat{\gamma}_N|\geq N^{-\kappa_1}\atop \text{or }|\gamma_N|\geq N^{-\kappa_2}}\prod_{i=1}^{N}\frac{1}{\sqrt{2\pi}}
e^{-\frac{\tilde{g}_i^2}{2}\left(1+\frac{2b(1+v-\lambda_i)x_2}{\sqrt{N}(1-2\theta \lambda_i+2\theta v)}\right)}\prod_{i=1}^{N}d\tilde{g}_i\right\}e^{-\frac{x_2^2}{2}}dx_2.
\end{align*}
To simplify it, we change the measure. Let
$h_i=\tilde{g}_i\sqrt{1+\frac{2b(1+v-\lambda_i)x_2}{\sqrt{N}(1-2\theta \lambda_i+2\theta v)}}$,
\begin{align}
\label{eq10}R\leq\int_{I_2}\frac{e^{bx_2\sqrt{N}}}{\prod_{i=1}^{N}\sqrt{1+\frac{2b(1+v-\lambda_i)x_2}{\sqrt{N}(1-2\theta \lambda_i+2\theta v)}}}P_h(\mathcal{D}^c) e^{-\frac{x_2^2}{2}}dx_2.
\end{align}
Here $P_h(\cdot)$ is the Gaussian measure of $\{h_i\}_{i=1}^{N}$. Take $0<\epsilon<\frac{1}{2}-\kappa_1$, we can separate the above integral into two parts,
\begin{align}
\label{eq10}R\leq\underbrace{\int_{[-N^{\epsilon},N^{\epsilon}]}\frac{e^{bx_2\sqrt{N}}}{\prod_{i=1}^{N}\sqrt{1+\frac{2b(1+v-\lambda_i)x_2}{\sqrt{N}(1-2\theta \lambda_i+2\theta v)}}}P_h(\mathcal{D}^c) e^{-\frac{x_2^2}{2}}dx_2}_{E_1}+\underbrace{\int_{I_2\cap[-N^{\epsilon},N^{\epsilon}]^c}\frac{e^{bx_2\sqrt{N}}}{\prod_{i=1}^{N}\sqrt{1+\frac{2b(1+v-\lambda_i)x_2}{\sqrt{N}(1-2\theta \lambda_i+2\theta v)}}} e^{-\frac{x_2^2}{2}}dx_2}_{E_2}.
\end{align}

For $E_2$, it is of the same form as (\ref{pr4}). Use the same argument, the main contribution from $E_2$ is a Gaussian integral on $[-N^{\epsilon},N^{\epsilon}]^c$. So it stretched exponential decays, $E_2\leq c'e^{-cN^{\zeta}}$. For $E_1$,
\begin{align*}
E_1\leq\int\frac{e^{bx_2\sqrt{N}}}{\prod_{i=1}^{N}\sqrt{1+\frac{2b(1+v-\lambda_i)x_2}{\sqrt{N}(1-2\theta \lambda_i+2\theta v)}}} e^{-\frac{x_2^2}{2}}dx_2\sup_{x_2\in [-N^{\epsilon},N^{\epsilon}]}\{P_h(\mathcal{D}^c) \}.
\end{align*}
If we can show that on the interval $[-N^{\epsilon}, N^{\epsilon}]$, $P_h(\mathcal{D}^c)$ is uniformly exponentially small, independent of $x_2$, then $E_1$ is exponentially small. For the upper bound of $P_h(\mathcal{D}^c)$, first by the union bound,
\begin{align*}
P_h(\mathcal{D}^c)\leq P_h(|\gamma_N|>N^{-\kappa_1})+P_h(|\hat{\gamma}_N|>N^{-\kappa_2}).
\end{align*}
For simplicity here I will only bound the first term, the second term can be bounded in exactly the same way.
\begin{align}
\notag P_h(|\gamma_N|>N^{-\kappa_1})=&P_h(|\frac{1}{N}\sum_{i=1}^{N}\frac{h_i^2}{1-2\lambda_i\theta+2\theta v+ \frac{2bx_2(1+v-\lambda_i)}{\sqrt{N}}}-1|\geq N^{-\kappa_1})\\
\label{defv}=&P_h(|\frac{1}{N}\sum_{i=1}^{N}\frac{h_i^2-1}{1-2\lambda_i\theta+2\theta v+ \frac{2bx_2(1+v-\lambda_i)}{\sqrt{N}}}+O(N^{\epsilon-\frac{1}{2}})|\geq N^{-\kappa_1})\\
\label{in1}\leq& c'e^{-cN^{\zeta}},
\end{align}
For (\ref{defv}), notice the definition of $v$, we have $\frac{1}{N}\sum_{i=1}^{N}\frac{1}{1-2\lambda_i\theta+2v\theta}=1$.
And for the last inequality, since $\epsilon-\frac{1}{2}<-\kappa_2$, the term $O(N^{\epsilon-\frac{1}{2}})$ is negligible compared with $N^{-\kappa_1}$. Then the concentration measure inequality in Lemma \ref{CMI} implies (\ref{in1}).
\end{proof}

\subsection{Higher Order Expansion}

To compute the higher order expansion of the spherical integral $I_N(\theta,B_N)$, we need to obtain a full asymptotic expansion of (\ref{eq4}).

\begin{align}
\label{exp3.0}&\int_{|\gamma_N|\leq N^{-\kappa_1},\atop |\hat{\gamma}_N|\leq N^{-\kappa_2}}\exp\{-\theta N\gamma_N\frac{\hat{\gamma}_N-v\gamma_N}{\gamma_N+1}\}
\prod_{i=1}^{i=N}dP(\tilde{g}_i)\\
\notag=&\int_{|\gamma_N|\leq N^{-\kappa_1}, \atop |\hat{\gamma}_N|\leq N^{-\kappa_2}}\exp\{-\theta N \gamma_N(\hat{\gamma}_N-v\gamma_N)\}\exp\{\theta N\gamma_N^2\frac{\hat{\gamma}_N-v\gamma_N}{\gamma_N+1}\}\prod_{i=1}^{N}dP(\tilde{g}_i)\\
\notag=&\int_{|\gamma_N|\leq N^{-\kappa_1}, \atop |\hat{\gamma}_N|\leq N^{-\kappa_2}}\exp\{-\theta N \gamma_N(\hat{\gamma}_N-v\gamma_N)\}\left\{\sum_{k=0}^{\infty}\frac{1}{k!}(\theta N\gamma_N^2\frac{\hat{\gamma}_N-v\gamma_N}{\gamma_N+1})^k\right\}\prod_{i=1}^{N}dP(\tilde{g}_i)\\
\label{exp2.1}=&\int_{|\gamma_N|\leq N^{-\kappa_1}, \atop |\hat{\gamma}_N|\leq N^{-\kappa_2}}\exp\{-\theta N \gamma_N(\hat{\gamma}_N-v\gamma_N)\}\prod_{i=1}^{N}dP(\tilde{g}_i)+\sum_{l=1}^{+\infty}
\int_{|\gamma_N|\leq N^{-\kappa_1}, \atop |\hat{\gamma}_N|\leq N^{-\kappa_2}}\exp\{-\theta N \gamma_N(\hat{\gamma}_N-v\gamma_N)\}\\
\notag&\times\left\{\sum_{k=1}^{l}{l-1\choose k-1}\frac{(-1)^k\theta^k}{k!}(-\gamma_N)^{l}(\sqrt{N}\gamma_N)^k\left(\sqrt{N}(\hat{\gamma}_N-v\gamma_N)\right)^k\right\}\prod_{i=1}^{N}dP(\tilde{g}_i)
\end{align}
We consider the $l$-th summand in the expression (\ref{exp2.1}). As proved in theorem $3$ of \cite{GuMa}, the distribution of $(\sqrt{N}\gamma_N, \sqrt{N}\hat{\gamma}_N)$ tends to $\Gamma$, which is centered two-dimensional Gaussian measure on $\mathbb{R}^{2}$ with covariance matrix
\begin{align*}
R=2
 \left[ \begin{array}{cc}
A_2 & F\\
F & G \end{array} \right]
\end{align*}
Moreover, the matrix $R$ is non-degenerate. Therefore, by the central limit theorem, we can obtain the asymptotic expression of the $l$-th term in (\ref{exp2.1})
\begin{align*}
&\int_{|\gamma_N|\leq N^{-\kappa_1}, \atop |\hat{\gamma}_N|\leq N^{-\kappa_2}}\exp\{-\theta N \gamma_N(\hat{\gamma}_N-v\gamma_N)\}\left\{\sum_{k=1}^{l}{l-1\choose k-1}\frac{(-1)^k\theta^k}{k!}(-\gamma_N)^{l}(\sqrt{N}\gamma_N)^k\left(\sqrt{N}(\hat{\gamma}_N-v\gamma_N)\right)^k\right\}\prod_{i=1}^{N}dP(\tilde{g}_i)\\
=&(\frac{-1}{\sqrt{N}})^l\left\{\int\exp\{-\theta N x(y-vx)\}\left\{\sum_{k=1}^{l}{l-1\choose k-1}\frac{(-1)^k\theta^k}{k!}x^{l+k}(y-vx)^k\right\}d\Gamma(x,y)+o(1)\right\}
=O(N^{-\frac{l}{2}})
\end{align*}
Therefore if we cut off the infinite sum (\ref{exp2.1}) at the $l$-th term, then the error terms are of magnitude $O(N^{-\frac{l+1}{2}})$. To obtain the full expansion, we need to understand each term in the expansion (\ref{exp2.1}).
\begin{align}
\label{eq2.2}\int_{|\gamma_N|\leq N^{-\kappa_1}, \atop |\hat{\gamma}_N|\leq N^{-\kappa_2}}\exp\{-\theta N \gamma_N(\hat{\gamma}_N-v\gamma_N)\}\left\{{l-1\choose k-1}\frac{(-1)^k\theta^k}{k!}(-\gamma_N)^{l}N^k(\gamma_N)^k(\hat{\gamma}_N-v\gamma_N)^k\right\}\prod_{i=1}^{N}dP(\tilde{g}_i).
\end{align}
Define
\begin{align}
\label{key}f_l(t)=\int_{|\gamma_N|\leq N^{-\kappa_1}, \atop |\hat{\gamma}_N|\leq N^{-\kappa_2}}\exp\{-t N \gamma_N(\hat{\gamma}_N-v\gamma_N)\}\gamma_N^l\prod_{i=1}^{N}dP(\tilde{g}_i).
\end{align}
By the dominated convergence theorem, we can interchange derivative and integral. The integral (\ref{eq2.2}) can be written as
\begin{align*}
(-1)^{l}{l-1\choose k-1}\frac{\theta^k}{k!}\frac{d^kf_l(t)}{dt^k}\Big{|}_{t=\theta}
\end{align*}
Therefore to understand the asymptotic expansion of (\ref{eq2.2}), we only need to compute the asymptotic expansion of $f_l(t)$, for $l=0,1,2,\cdots$. We have the following proposition
\begin{proposition}\label{pp2}
If $\sup_{N}\|B_N\|_{\infty}<M$, then for any $\theta\in \mathbb{R}$ such that $|\theta|<\frac{1}{4M^2+10M+1}$, $f_l$ has the following asymptotic expansion (up to $O(\frac{1}{N^{n+1}})$ for any given $n$)
\begin{align}
\label{exp3.1}f_l(t)=m_0+\frac{m_1}{N}+\frac{m_2}{N^2}+\cdots +\frac{m_n}{N^n}+O(\frac{1}{N^{n+1}})
\end{align}
where $\{m_i\}_{i=0}^{n}$ depends explicitly on $t$, $\theta$, $v$ and the derivative of the Hilbert transform of the empirical spectral distribution of $B_N$, namely, $A_2,A_3,A_4\cdots A_{2n+2}$.
\end{proposition}
\begin{proof}
First we show that $f_l$ has asymptotic expansion in form (\ref{exp3.1}), then we show those $m_i$'s depend only explicitly on $t$, $\theta$ and $\{A_k\}_{k=2}^{+\infty}$. Use the same method as for the first order term, we obtain the following expression of $f_l(t)$,
\begin{align*}
f_l(t)=\frac{1}{(\sqrt{2\pi})^2}\int_{I_1\times I_2}e^{ibx_1\sqrt{N}+bx_2\sqrt{N}} \left\{\int_{|\hat{\gamma}_N|\leq N^{-\kappa_1}\atop |\gamma_N|\leq N^{-\kappa_2}}\prod_{i=1}^{N}\frac{1}{\sqrt{2\pi}}e^{-\frac{\tilde{g}_i^2}{2}\left(1+\frac{2b(i(1-v+\lambda_i) x_1+(1+v-\lambda_i)x_2)}{\sqrt{N}(1-2\theta \lambda_i+2\theta v)}\right)}\right.\\
\left.\left(\frac{1}{N}\sum_{i=1}^{N}\frac{(\tilde{g}_i^2-1)}{1-2\theta\lambda_i+2\theta v}\right)^l \prod_{i=1}^{N}d\tilde{g}_i\vphantom{\int_{|\hat{\gamma}_N|\leq N^{-\kappa_1}\atop |\gamma_N|\leq N^{-\kappa_2}}}\right\}\prod_{i=1}^{2}e^{-\frac{x_i^2}{2}}dx_i+O(N^{-\infty})
\end{align*}
where $b^2=\frac{t}{2}$. The same argument about replacing integral domain in the inner integral can be implemented here without too much change. So we can replace the integral domain $\{{|\hat{\gamma}_N|\leq N^{-\kappa_1}, |\gamma_N|\leq N^{-\kappa_2}}\}$ by $\mathbb{R}^{N}$.
\begin{align}
\notag&e^{ibx_1\sqrt{N}+bx_2\sqrt{N}} \int_{\mathbb{R}^N}\prod_{i=1}^{N}\frac{1}{\sqrt{2\pi}}e^{-\frac{\tilde{g}_i^2}{2}
\left(1+\frac{2b(i(1-v+\lambda_i) x_1+(1+v-\lambda_i)x_2)}{\sqrt{N}(1-2\theta \lambda_i+2\theta v)}\right)}\{\frac{1}{N}\sum_{i=1}^{N}\frac{(\tilde{g}_i^2-1)}{1-2\theta\lambda_i+2\theta v}\}^l\prod_{i=1}^{N}d\tilde{g}_i\\
\notag=&\prod_{i=1}^{N} \frac{\exp\{\frac{b(i(1-v+\lambda_i) x_1+(1+v-\lambda_i)x_2)}{\sqrt{N}(1-2\theta \lambda_i+2\theta v)}\}}{\sqrt{1+\frac{2b(i(1-v+\lambda_i) x_1+(1+v-\lambda_i)x_2)}{\sqrt{N}(1-2\theta \lambda_i+2\theta v)}}} \int_{\mathbb{R}^N}\prod_{i=1}^{N}\frac{1}{\sqrt{2\pi}}e^{-\frac{\tilde{g}_i^2}{2}}
\left\{\frac{1}{N}\sum_{i=1}^{N} \frac{\frac{\tilde{g}_i^2}{1+\frac{2b(i(1-v+\lambda_i) x_1+(1+v-\lambda_i)x_2)}{\sqrt{N}(1-2\theta \lambda_i+2\theta v)}}-1}{1-2\theta\lambda_i+2\theta v}\right\}^l\prod_{i=1}^{N}d\tilde{g}_i\\
\label{exp3.2}=&\underbrace{\prod_{i=1}^{N} \frac{e^{\mu_i}}{\sqrt{1+2\mu_i}} }_{E_1}\underbrace{\int\prod_{i=1}^{N}\frac{1}{\sqrt{2\pi}}e^{-\frac{\tilde{g}_i^2}{2}}\left\{\frac{1}{N}\sum_{i=1}^{N} \frac{\nu_i}{1+2\mu_i}(\tilde{g}_i^2-(1+2\mu_i))\right\}^l\prod_{i=1}^{N}d\tilde{g}_i}_{E_2}
\end{align}
where the equations come from change of variables (by Lemma \ref{CI}), and
\begin{align*}
\nu_i=\frac{1}{1-2\theta\lambda_i+2\theta v},\quad \mu_i=\frac{b(i(1-v+\lambda_i) x_1+(1+v-\lambda_i)x_2)}{\sqrt{N}(1-2\theta \lambda_i+2\theta v)}.
\end{align*}
Notice here $\mu_i$ can be written as a linear function of $\nu_i$
\begin{align}
\label{rel1}\mu_i=\left(i(1+\frac{1}{2\theta})\frac{x_1}{\sqrt{N}}+(1-\frac{1}{2\theta})\frac{x_2}{\sqrt{N}}\right)b\nu_i
+(\frac{x_2}{\sqrt{N}}-i\frac{x_1}{\sqrt{N}})\frac{b}{2\theta}.
\end{align}
The formula (\ref{exp3.2}) consists of two parts: a product factor $E_1$ and a Gaussian integral $E_2$. For $E_1$ we can obtain the following explicit asymptotic expansion
\begin{align}
\prod_{i=1}^{N}\frac{e^{\mu_i}}{\sqrt{1+2\mu_i
}}
\notag=&\prod_{i=1}^{N}e^{\mu_i
-\frac{1}{2}\log(1+2\mu_i)}
=\prod_{i=1}^{N}e^{\mu_i
^2+\sum_{k=3}^{\infty}\frac{(-1)^k 2^{k-1}}{k}\mu_i^{k}}\\
\label{exp3.3}=&e^{\sum_{i=1}^{N}\mu_i
^2}e^{\sum_{k=3}^{\infty}\frac{1}{N^{\frac{k}{2}-1}}\left\{\frac{(-1)^k 2^{k-1}}{k}\sum_{i=1}^{N}\frac{b^k(i(1-v+\lambda_i) x_1+(1+v-\lambda_i)x_2)^k}{N(1-2\theta \lambda_i+2\theta v)^k}\right\}}.
\end{align}
Notice
\begin{align*}
&\sum_{i=1}^{N}\frac{(i(1-v+\lambda_i) x_1+(1+v-\lambda_i)x_2)^k}{N(1-2\theta \lambda_i+2\theta v)^k}\\
=&\sum_{m=0}^{k}\left\{{k\choose m}(ix_1-x_2)^m(i(1-v)x_1+(1+v)x_2)^{k-m}\sum_{i=1}^N\frac{\lambda_i^m}{N(1+2\theta\lambda_i+2\theta v)^k}\right\}.
\end{align*}
If we regard $v$ and $\theta$ as constant (since they are of magnitude $O(1)$), then $\sum_{i=1}^N\frac{\lambda_i^m}{N(1+2\theta\lambda_i+2\theta v)^k}$ can be written as a linear combination of $A_2,A_3\cdots ,A_k$ for any $0\leq m\leq k$. Thus we can expand (\ref{exp3.3}), to obtain
\begin{align*}
\prod_{i=1}^{N}\frac{e^{\mu_i}}{\sqrt{1+2\mu_i
}}=e^{\sum_{i=1}^{N} \mu_i^2}\left\{1+\sum_{k=1}^{\infty}\frac{1}{N^{\frac{k}{2}}}g_k(x_1, x_2)\right\},
\end{align*}
where $g_k(x_1,x_2)$'s are polynomials of $x_1$ and $x_2$. Consider $t$, $\theta$ and $v$ as constant, the coefficients of $g_k(x_1,x_2)$ are polynomials in terms of $A_2,A_3,\cdots A_{k+2}$. Moreover the degree of each monomial of $g_k(x_1,x_2)$ is congruent to $k$ modula $2$.

Next we compute the Gaussian integral $E_2$ in (\ref{exp3.2}). Expand the $l$-th power, we obtain
\begin{align*}
&\left\{\frac{1}{N}\sum_{i=1}^{N} \frac{\nu_i}{1+2\mu_i}(\tilde{g}_i^2-(1+2\mu_i))\right\}^l\\
&\qquad=\frac{1}{N^{l}}\sum_{k_1\geq k_2\cdots\geq k_m\atop k_1+k_2\cdots k_m=l}\left\{\frac{l!}{k_1!k_2!\cdots k_m!}\sum_{1\leq i_1,i_2,\cdots i_m\leq N\atop \text{distinct}}\prod_{j=1}^{m}(\frac{\nu_{i_j}}{1+2\mu_{i_j}})^{k_j}(\tilde{g}_{i_j}^2-(1+2\mu_{i_j}))^{k_j}\right\}.
\end{align*}
Denote
\begin{align*}
p_{k_j(\mu_{i_j})}=\int\frac{1}{\sqrt{2\pi}}e^{-\frac{\tilde{g}_{i_j}^2}{2}}(\tilde{g}_{i_j}^2-(1+2\mu_{i_j}))^{k_j}d\tilde{g}_{i_j},
\end{align*}
then $p_{k_j}$ is a $k_j$-th degree polynomial, which only depends on $k_j$. With this notation, the Gaussian integral $E_2$ can be written as,
\begin{align*}
\frac{1}{N^{l}}\sum_{k_1\geq k_2\cdots\geq k_m\atop k_1+k_2\cdots k_m=l}\left\{\frac{l!}{k_1!k_2!\cdots k_m!}\sum_{1\leq i_1,i_2,\cdots i_m\leq N\atop \text{distinct}}\prod_{j=1}^{m}(\frac{\nu_{i_j}}{1+2\mu_{i_j}})^{k_j}p_{k_j}(\mu_{i_j})\right\}.
\end{align*}
By the following lemma, the above expression can be expressed in a more symmetric way, as a sum in terms of
\begin{align}
\label{exp3.4}\frac{1}{N^{l-m}}\prod_{j=1}^{m}\left\{\frac{1}{N}\sum_{i=1}^{N}(\frac{\nu_i}{1+\mu_i})^{k_j}q_{k_j}(\mu_i)\right\},
\end{align}
where $k_1\geq k_2 \cdots \geq k_m$, $k_1+k_2+\cdots k_m=l$ and $q_{k_j}$'s are some polynomials depending only on $k_j$.

\begin{lemma}
Given integers $s_1,s_2\cdots, s_m$ and polynomials $q_1,q_2,\cdots q_m$, consider the following polynomial in terms of $2N$ variables $x_1,x_2\cdots x_N, y_1,y_2\cdots y_N$
\begin{align*}
h=\sum_{1\leq i_1,i_2\cdots i_m\leq N\atop \text{distinct}}\prod_{j=1}^{m}x_{i_j}^{s_j}q_j(y_{i_j}).
\end{align*}
Then $h$ can be expressed as sum of terms in the following form
\begin{align}
\label{form}\prod_{j=1}^{l}\{\sum_{i=1}^{N}x_i^{t_j}\tilde{q}_j(y_i)\},
\end{align}
where $l$, $\{t_i\}_{i=1}^{N}$ and polynomials $\{\tilde{q}_i\}_{i=1}^{N}$ are to be chosen.
\end{lemma}
\begin{proof}
We prove this by induction on $m$. If $m=1$ then $h$ itself is of the form (\ref{form}). We assume the statement holds for $1,2,3\cdots, m-1$, then we prove it for $m$.
\begin{align}
\notag &h-\prod_{j=1}^{m}(\sum_{i=1}^{N}x_i^{s_j}q_j(y_i))\\
\label{sum}=& -\sum_{d=1}^{m-1}\sum_{\pi_1,\pi_2\cdots \pi_d\atop \text{a partition of }\{1,2,\cdots, m\}}\sum_{1\leq i_1,i_2\cdots,i_d\leq N\atop \text{distinct}}\prod_{j=1}^{d}x_{i_j}^{\sum_{l\in \pi_j}s_l}\prod_{l\in \pi_j}q_{l}(y_{i_j}).
\end{align}
Notice the summands of (\ref{sum}) are $\sum_{1\leq i_1,i_2\cdots,i_d\leq N\atop \text{distinct}}\prod_{j=1}^{d}x_{i_j}^{\sum_{l\in \pi_j}s_l}\prod_{l\in \pi_j}q_{l}(y_{i_j})$, which are of the same form as $h$ but with less $m$. Thus, by induction, each term in (\ref{sum}) can be expressed as sum of terms in form (\ref{form}), so does $h$.
\end{proof}

In view of (\ref{exp3.4}), since we have $\mu_i=O(N^{\epsilon-\kappa_2})$, we can Taylor expand $\frac{1}{1+\mu_i}$ in (\ref{exp3.4}). Also notice (\ref{rel1}), the relation between $\mu_i$ and $\nu_i$, (\ref{exp3.4}) has the following full expansion
\begin{align}
\label{exp3.5}\frac{1}{N^{l-m}}\sum_{k=0}^{\infty}\frac{1}{N^{\frac{k}{2}}}h_k(x_1,x_2),
\end{align}
where $h_k(x_1,x_2)$ are $k$-th degree polynomials of $x_1$ and $x_2$. Consider $t$, $\theta$ and $v$ as constant (since they are of magnitude $O(1)$), the coefficients of $h_k(x_1,x_2)$ are polynomials of $A_2,A_3,\cdots A_{k+l-m+1}$. Since the Gaussian integral $E_2$ is the sum of terms which has the asymptotic expansion (\ref{exp3.5}), itself has the full asymptotic expansion,
\begin{align}
E_2=\label{exp3.6}\sum_{k=0}^{\infty}\frac{1}{N^{\frac{k}{2}}}s_k(x_1,x_2),
\end{align}
where the coefficients of $s_k(x_1,x_2)$ are polynomials of $A_2,A_3,\cdots A_{k+1}$. And the degree of each monomial of $s_k$ is congruent to $k$ modula $2$. Combine the asymptotic expansions of $E_1$ and $E_2$, we obtain the expansion of $f_l(t)$,
\begin{align}
\notag f_l(t)=&\frac{1}{2\pi}\int_{I_1\times I_2}e^{-\frac{1}{2}\langle x, \tilde{K} x\rangle}\left\{\sum_{k=0}^{\infty}\frac{1}{N^{\frac{k}{2}}}g_k(x_1,x_2)\right\}\{\sum_{k=0}^{\infty}\frac{1}{N^{\frac{k}{2}}}s_k(x_1,x_2)\}dx_1dx_2+O(N^{-\infty})\\
\label{exp3.6}=&\frac{1}{2\pi}\int_{I_1\times I_2}e^{-\frac{1}{2}\langle x, \tilde{K} x\rangle}\left\{\sum_{k=0}^{\infty}\frac{\sum_{l=0}^{k}g_l(x_1,x_2)s_{k-l}(x_1,x_2)}{N^{\frac{k}{2}}}\right\}dx_1dx_2+O(N^{-\infty}).
\end{align}
where $\tilde{K}$ is the following $2\times 2$ matrix
\begin{align*}
K=
 \left[ \begin{array}{cc}
1+t\left((1-v)^2A_2+2(1-v)F+G\right) & -t i\left((1-v^2)A_2+2vF-G\right) \\
-t i\left((1-v^2)A_2+2vF-G\right) & 1-t\left((1+v)^2A_2-2(1+v)F+G\right)  \end{array} \right]
\end{align*}
The formula (\ref{exp3.6}) is a Gaussian integral in terms of $x_1$ and $x_2$. If we cut off at $k=m$, this will result in an error term $O(N^{-\frac{m+1}{2}})$. Now the integrand is a finite sum. The integral is $O(N^{-\infty})$ outside the region $I_1\times I_2$, thus we obtain the following asymptotic expansion,
\begin{align}
\label{exp3.7}f_l(t)=
\sum_{k=0}^{m}\frac{1}{N^{\frac{k}{2}}}\frac{1}{2\pi}\int e^{-\frac{1}{2}\langle x, \tilde{K} x\rangle}\left\{\sum_{l=0}^{k}g_l(x_1,x_2)s_{k-l}(x_1,x_2)\right\}dx_1dx_2+O(N^{-\frac{m+1}{2}}).
\end{align}
Notice the degree of each monomial of $g_l(x_1,x_2)$ is congruent to $l$ modula $2$, and the degree of each monomial of $s_{k-l}(x_1,x_2)$ is congruent to $k-l$ modula $2$. For any odd $k$,  $\sum_{l=0}^{k}g_l(x_1,x_2)s_{k-l}(x_1,x_2)$ is sum of monomials of odd degree. Since here $x_1$, $x_2$ are centered Gaussian variables, the integral of $\sum_{l=0}^{k}g_l(x_1,x_2)s_{k-l}(x_1,x_2)$ vanishes. Using Lemma \ref{GIF}, (\ref{exp3.7}) can be rewritten as
\begin{align*}
f_l(t)=&
\sum_{k=0}^{[\frac{m}{2}]}\frac{1}{N^{k}}\frac{1}{2\pi}\int e^{-\frac{1}{2}\langle x, \tilde{K} x\rangle}\{\sum_{l=0}^{2k}g_l(x_1,x_2)s_{2k-l}(x_1,x_2)\}dx_1dx_2+O(N^{-[\frac{m}{2}]-1})\\
=&\sum_{k=0}^{[\frac{m}{2}]}\frac{1}{N^{k}}\frac{1}{\sqrt{\det(\tilde{K})}}\{\sum_{l=0}^{2k}g_l(\frac{\partial}{\partial \xi_1},\frac{\partial}{\partial \xi_2})s_{2k-l}(\frac{\partial}{\partial \xi_1},\frac{\partial}{\partial \xi_2})e^{\frac{1}{2}\langle \xi, \tilde{K}^{-1}\xi\rangle}\}\Big{|}_{\xi=0}+O(N^{-[\frac{m}{2}]-1}).
\end{align*}
Since the entries of matrix $\tilde{K}$ and the coefficients of $\sum_{l=0}^{2k}g_l(x_1,x_2)s_{2k-l}(x_1,x_2)$ depend only on $t$, $\theta$, $v$ and $A_2,A_3,\cdots,A_{2k+2}$. This implies $f_l(t)$ has the expansion (\ref{exp3.1}).
\end{proof}
From the argument above, the asymptotic expansion of (\ref{exp3.0}) up to $O(N^{-\frac{l}{2}})$ is a finite sum in terms of derivatives of $f_k(t)$'s at $t=\theta$. Differentiate $f_l(t)$ term by term, we arrive at our main theorem of this paper,

\begin{theorem}\label{rh}
If $\sup_{N}\|B_N\|_{\infty}<M$, then for any $\theta\in \mathbb{R}$ such that $|\theta|<\frac{1}{4M^2+10M+1}$, the spherical integral has the following asymptotic expansion (up to $O(\frac{1}{N^{n+1}})$ for any given $n$)
\begin{align*}
e^{-\frac{N}{2}\int_0^{2\theta} R_{B_N}(s)ds}I_N(\theta, B_N)
=m_0+\frac{m_1}{N}+\frac{m_2}{N^2}+\cdots +\frac{m_n}{N^n}+O(\frac{1}{N^{n+1}})
\end{align*}
where $R_{B_N}$ is the R-transform of empirical spectral distribution of $B_N$,  and $\{m_i\}_{i=0}^{n}$ depend on $\theta$, $v$ and the derivatives of Hilbert transform of the empirical spectral distribution of $B_N$ , namely $A_2,A_3,\cdots A_{2n+2}$.
Especially we have
\begin{align*}
m_0=\frac{1}{\sqrt{A_2}},\quad m_1=\frac{1}{\sqrt{A_2}}(\frac{3}{2}\frac{A_4}{A_2^2}-\frac{5}{3}\frac{A_3^2}{A_2^3}+\frac{1}{6}).
\end{align*}
\end{theorem}
\begin{proof}
In the last section, we have computed the first term in the expansion $m_0=\frac{1}{\sqrt{A_2}}$. We only need to figure out the second term $m_1$. For this we cut off (\ref{exp3.1}) at $l=2$,
\begin{align}
\notag&\int_{|\gamma_N|\leq N^{-\kappa_1},\atop |\hat{\gamma}_N|\leq N^{-\kappa_2}}\exp\{-\theta N\gamma_N\frac{\hat{\gamma}_N-v\gamma_N}{\gamma_N+1}\}
\prod_{i=1}^{i=N}dP(\tilde{g}_i)\\
\label{exp4.1}=&f_0\Big{|}_{t=\theta}-\theta \frac{d}{dt}f_1\Big{|}_{t=\theta}+(\frac{\theta^2}{2}
\frac{d^2}{dt^2}+\theta\frac{d}{dt})f_2\Big{|}_{t=\theta}+O(N^{-2}).
\end{align}
Take $l=0,1,2$ in (\ref{key}), we obtain the asymptotic expansion of $f_0$, $f_1$ and $f_2$ (we put the detailed computation in the appendix),
\begin{align*}
f_0\Big{|}_{t=\theta}=&\frac{1}{\sqrt{A_2}}
+\frac{1}{N}\frac{1}{\sqrt{A_2}}(\frac{7}{6}-\frac{3}{A_2}
+\frac{2A_3}{A_2^2}-\frac{5}{3}\frac{A_3^2}{A_2^3}+\frac{3}{2}\frac{A_4}{A_2^2})+O(N^{-2}),\\
-\theta \frac{d}{dt}f_1\Big{|}_{t=\theta}=&\frac{1}{N}\frac{1}{\sqrt{A_2}}(-4+\frac{6}{A_2}-\frac{2A_3}{A_2^2})+O(N^{-2}),\\
(\frac{\theta^2}{2}
\frac{d^2}{dt^2}+\theta\frac{d}{dt})f_2\Big{|}_{t=\theta}=&\frac{1}{N}\frac{1}{\sqrt{A_2}}(3-\frac{3}{A_2})+O(N^{-2}).
\end{align*}
Plug them back to (\ref{exp4.1}), we get
\begin{align*}
m_0=\frac{1}{\sqrt{A_2}},\quad m_1=\frac{1}{\sqrt{A_2}}(\frac{3}{2}\frac{A_4}{A_2^2}-\frac{5}{3}\frac{A_3^2}{A_2^3}+\frac{1}{6}).
\end{align*}
\end{proof}

\section{Unitary Case}{\label{sec4}}

In this section we consider the unitary case, $\beta=2$. As we will see soon that the unitary case is a special case of orthogonal case. With the same notation as before, let $B_{N}=\text{diag}(\lambda_1,\lambda_2,\cdots, \lambda_N)$. and $U$ follows the Haar measure on unitary group $U(N)$. The first column $e_1$ of $U$ can be parametrized as the normalized complex Gaussian vector,
\begin{align*}
e_1=\frac{g^{(1)}+ig^{(2)}}{\|g^{(1)}+ig^{(2)}\|},
\end{align*}
where $g^{(1)}=(g_1,g_3\cdots,g_{2N-3},g_{2N-1})^{T}$ and $g^{(2)}=(g_{2},g_{4},\cdots, g_{2N-2},g_{2N})^{T}$ are independent Gaussian vectors in $\mathbb{R}^N$. Then the spherical integral has the following form
\begin{align*}
I_N^{(2)}(\theta,B_N)=\int \exp\left\{N\theta\frac{\lambda_1(g_1^2+g_{2}^2)+\lambda_2(g_3^3+g_{4}^2)+\cdots +\lambda_N(g_{2N-1}^2+g_{2N}^2)}{g_1^2+g_2^2+\cdots+g_{2N-1}^2+g_{2N}^2}\right\}\prod_{i=1}^{2N}dP(g_i).
\end{align*}
Consider the $2N\times 2N$ diagonal matrix $D_{2N}=\text{diag}\{\lambda_1,\lambda_1,\lambda_2,\lambda_2,\cdots,\lambda_N,\lambda_N\}$ with each $\lambda_i$ appearing twice. Then we have the following relation,
\begin{align*}
I_{N}^{(2)}(\theta, B_N)=I_{2N}^{(1)}(\frac{\theta}{2},D_{2N}).
\end{align*}
Define $v$ and $\{A_i\}_{i=1}^{\infty}$ as in the notation section but replace $\theta$ by $\frac{\theta}{2}$ and replace $B_N$ by $D_{2N}$,
namely, $v=R_{B_N}(\theta)$ and
\begin{align*}
A_k=\frac{(-1)^{k-1}}{(k-1)!\theta^{k}}\frac{d^{k-1}H_{B_N}}{dz^{k-1}}(v+\frac{1}{\theta})
=\frac{1}{N}\sum_{i=1}^{N}\frac{1}{(1-\theta\lambda_i+\theta v)^k}.
\end{align*}
Then from Theorem \ref{rh}, we have the following theorem for unitary case.
\begin{theorem}
If $\sup_{N}\|B_N\|_{\infty}<M$, then for any $\theta\in \mathbb{R}$ such that $|\theta|<\frac{1}{4M^2+10M+1}$, the spherical integral $I_{N}^{(2)}(\theta, B_N)$ has the following asymptotic expansion (up to $O(\frac{1}{N^{n+1}})$ for any given $n$)
\begin{align*}
I_{N}^{(2)}(\theta, B_N)=I_{2N}^{(1)}(\frac{\theta}{2},D_{2N})
=e^{N\int_{0}^{\theta} R_{B_N}(s)ds}\left\{m_0+\frac{m_1}{N}+\frac{m_2}{N^2}+\cdots+\frac{m_n}{N^n}+O(\frac{1}{N^{n+1}})\right\}
\end{align*}
where $R_{B_N}$ is the R-transform of the empirical spectral distribution of $B_N$, and $\{m_i\}_{i=0}^{n}$ depend on $\theta$, $v$, $\{A_i\}_{i=2}^{2n+2}$, and the derivatives of Hilbert transform of the empirical spectral distribution of $B_N$ at $v+\frac{1}{\theta}$.
Especially we have
\begin{align*}
m_0=\frac{1}{\sqrt{A_2}},\quad m_1=\frac{1}{2\sqrt{A_2}}(\frac{3}{2}\frac{A_4}{A_2^2}-\frac{5}{3}\frac{A_3^2}{A_2^3}+\frac{1}{6}).
\end{align*}
\end{theorem}

\section{Asymptotic Free Convolution}{\label{sec5}}
In this section we use spherical integral to derive the asymptotic free convolution of Hermitian matrix (deterministic) and Wigner matrix. We consider the real Wigner matrix $X_N$ with entries (the complex case can be proved in the same way). Recall that a Wigner matrix is a symmetric random matrix $X_N$ such that
\begin{itemize}
\item the subdiagonal entries of $X_N$ are independent and identically distributed.
\item the random variables $\sqrt{N}X_N(i,j)$ are distributed according to a measure $\mu$ independent of $N$, and of finite moments.
\end{itemize}
Further more we make the following assumption on the law $\mu$: The probability measure $\mu$ has mean zero, variance one and satisfies a log-Sobolev inequality with coefficient $m$. With this assumption, the measure $\mu$ has exponential decay, its Laplace transform is well defined, even more we have the following bound for its Laplace transform,

\begin{lemma}{\label{lemp}}
For any real valued random variable $X$ with zero mean, variance one and satisfying log-Sobolev inequality with coefficient $m$, we have the following bounds 
\begin{align*}
e^{\frac{t^2}{2}-c|t|^3}\leq \mathbb{E}_{X}[e^{tX}]\leq e^{\frac{t^2}{2}+c|t|^3}, \quad \forall t\in[-1,1].
\end{align*}
where $c$ is some constant depending on $m$.
\end{lemma}

\begin{theorem}\label{thm1}
Let $A_N$ be a sequence of uniformly bounded deterministic real Hermitian matrices with empirical eigenvalue distribution $\mu_{A_N}$ converges weakly to a compactly supported measure $\mu_A$ fast enough. Let $X_N$ be a sequence of Wigner matrices, satisfying log-Sobolev inequality. Then we have
\begin{align}
\label{eqthm}\lim_{N\rightarrow \infty} \frac{1}{N}\log I_N(\theta, A_N+X_N)=\lim_{N\rightarrow \infty} \frac{1}{N}\log I_N(\theta, A_N)+\lim_{N\rightarrow \infty} \frac{1}{N} \log I_N(\theta, X_N), a.s.
\end{align}
\end{theorem}

The proof follow the ideas of Section (6) of \cite{GuMa}, where Guionnet and Ma\"{i}da
proved the asymptotic free convolution of two independent symmetric (respectively Hermitian) matrices such that at least one of them is invariant under conjugation of orthogonal (respectively unitary) matrix. We here generalize their approach to Wigner matrices.

The proof consists of two steps, first we will show that $\frac{1}{N}\log \mathbb{E}[I_N(\theta, A_N+X_N)]$ splits into two terms, one corresponds to the spherical integral of $A_N$ and another corresponds to the spherical integral of $X_N$ (Proposition \ref{pro1}). Then we will show that we can exchange integration with the logarithm (Proposition \ref{pro2}).
\begin{proposition}\label{pro1}
Let $(A_N, X_N)_{N\in \mathbb{N}}$ be a sequence of deterministic real Hermitian matrices and Wigner matrices as in Theorem (\ref{thm1}). Then we have
\begin{align}\label{eq1}
\lim_{N\rightarrow \infty}\frac{1}{N}\log\mathbb{E}_{X_N}[I_N(\theta, A_N+X_N)]=\lim_{N\rightarrow \infty}\left[\frac{1}{N}\log I_N(\theta, A_N)
+\frac{1}{N}\log I_N(\theta, X_N)\right].
\end{align}
\end{proposition}

\begin{proposition}\label{pro2}
Let $(A_N, X_N)_{N\in \mathbb{N}}$ be a sequence of deterministic real Hermitian matrices and Wigner matrices as in Theorem (\ref{thm1}). Then we have
\begin{align*}
\lim_{N\rightarrow \infty}\frac{1}{N}\log I_N(\theta, A_N+X_N)=\lim_{N\rightarrow \infty}\frac{1}{N}\log\mathbb{E}_{X_N}[I_N(\theta, A_N+X_N)]\quad  a.s.
\end{align*}
\end{proposition}

Theorem (\ref{thm1}) is exactly the combination of Proposition \ref{pro1} and Proposition \ref{pro2}. Notice the relation between spherical integral and $R$-transform,
\begin{align*}
\lim_{N\rightarrow \infty}\frac{1}{N}\log I_N(\theta, A_N+X_N)=\lim_{N\rightarrow \infty}\frac{1}{2}\int_0^{2\theta}R_{A_N+X_N}(v)dv,
\end{align*}
Replace them in (\ref{eqthm}), we get
\begin{align*}
\lim_{N\rightarrow \infty}\frac{1}{2}\int_{0}^{2\theta}R_{A_N+X_N}(v)dv
=\lim_{N\rightarrow \infty}\frac{1}{2}\int_{0}^{2\theta}R_{A_N}(v)dv
+\lim_{N\rightarrow \infty}\frac{1}{2}\int_{0}^{2\theta}vdv\quad a.s.
\end{align*}
If we differentiate both sides, the asymptotic freeness of $A_N$ and $X_N$ follows.

In the following we give the detailed proof for Proposition \ref{pro1}. For Proposition \ref{pro2}, it is analogue to section 6 of \cite{GuMa}, we only list the key ingredient for the proof.
\begin{proof}
The spherical integral can be written as a double integral
\begin{align*}
\mathbb{E}_{X_N}[I_N(\theta, A_N+X_N)]
=&\mathbb{E}_{X_N}\left[\int e^{\theta N \langle e, A_Ne\rangle} e^{\theta N \langle e, X_N e\rangle}d\mathbb{P}(e)\right]\\
=&\int e^{\theta N\langle e, A_N e\rangle} \mathbb{E}_{X_N}[e^{\theta N \langle e, X_N e\rangle}] d\mathbb{P}(e).
\end{align*}
We want to split the above integral into two parts one for $I_N(\theta, A_N)$ and one for $I_N(\theta, X_N)$. To show this we prove that asymptotically $\mathbb{E}_{X_N}[e^{\theta N \langle e, X_N e\rangle}]$ does not depend on $e$, in fact it is asymptotically $e^{\theta^2 N}$. Thus the integral is asymptotically the multiplication of two term $\int e^{\theta N\langle e, A_N e\rangle} d\mathbb{P}(e) \times \mathbb{E}_{X_N}[e^{\theta N \langle e, X_N e\rangle}]$. To do this we need some concentration property of the vector $e$,
\begin{lemma}
Vector $e$ following the uniform measure on the sphere. Then we have the concentration of measure inequality, for any $0<\epsilon<\frac{1}{2}$,
\begin{align*}
\mathbb{P}(\max_{1\leq i \leq N}|e_i|> \frac{1}{N^{\frac{1}{2}-\epsilon}})\leq Ne^{-cN^{2\epsilon}}.
\end{align*}
\end{lemma}
\begin{proof}
\begin{align*}
\mathbb{P}\left(\max_{1\leq i \leq N}|e_i|> \frac{1}{N^{\frac{1}{2}-\epsilon}}\right)
\leq& N\mathbb{P}\left(\frac{|g_1|}{\sqrt{\frac{\sum g_i^2}{N}}}> N^{\epsilon}\right)\\
=&N\left[\mathbb{P}\left(\sqrt{\frac{\sum g_i^2}{N}}<2\right)+\mathbb{P}\left(|g_1|>N^\epsilon \sqrt{\frac{\sum g_i^2}{N}}, \sqrt{\frac{\sum g_i^2}{N}}>2\right)\right]\\
\leq & N\left[\mathbb{P}\left(\sum g_i^2<4N \right)+\mathbb{P}\left(|g_1|>2N^\epsilon\right)\right]
\leq  N e^{-cN^{2\epsilon}},
\end{align*}
for some constants $c$ depending on $\epsilon$.
\end{proof}
Define $I_{\epsilon}=\{\max_{1\leq i \leq N}|e_i|\leq\frac{1}{N^{\frac{1}{2}-\epsilon}}\}$. Then the integral of the spherical integral outside region $I_{\epsilon}$ is negligible. In fact we have the following proposition,
\begin{proposition}{\label{pro7}}
Given sequence of uniformly bounded Hermitian matrices $\{B_N\}_{N\in \mathbb{N}}$, i.e. $\sup_{N}\|B_N\|_{\infty}=M$, we have
\begin{align*}
\int_{I_{\epsilon}} e^{\theta N\langle e, B_N e\rangle}d\mathbb{P}(e)=(1+\eta(\epsilon, M, N))\int e^{\theta N \langle e, B_N e\rangle}d\mathbb{P}(e),
\end{align*}
where $\eta(\epsilon, M, N)$ goes to $0$, as $N$ goes to infinity.
\end{proposition}
\begin{proof}
From Proposition \ref{proposition1}, we have
\begin{align*}
\int_{\mathcal{D}}e^{\theta N\langle e, B_N e\rangle}d\mathbb{P}(e)
=(1+\delta(\kappa_1,\kappa_2, N))\int e^{\theta N\langle e, B_N e\rangle}d\mathbb{P}(e),
\end{align*}
where $\mathcal{D}=\{|\gamma_N|\leq N^{-\kappa_1}, |\hat{\gamma}_N|\leq N^{-\kappa_2}\}$, $\gamma_N$ and $\hat{\gamma}_N$ are defined in Section \ref{sec3}. $\delta(\kappa_1,\kappa_2, N)$ is some constant, and goes to $0$ as $N$ goes to infinity. Therefore the proposition holds if we can show that there exists constants $\eta(\kappa_1,\kappa_2,\epsilon, N))$, which goes to $0$ as $N$ goes to zero and,
\begin{align}\label{eq4.2}
\int_{I_{\epsilon}^{c}\cap\mathcal{D}}e^{\theta N\langle e, B_N e\rangle}d\mathbb{P}(e)
=\eta(\kappa_1,\kappa_2,\epsilon, N))\int_{\mathcal{D}} e^{\theta N\langle e, B_N e\rangle}d\mathbb{P}(e).
\end{align}
Since the spherical integral can be written as
\begin{align}\label{eq4.3}
\int e^{\theta N\langle e, B_N e\rangle}d\mathbb{P}(e)=e^{\frac{N}{2}\int_0^{2\theta} R_{B_N}(s)ds}\int \exp\left\{-\theta N\frac{\gamma_N(\hat{\gamma}_N-v\gamma_N)}{\gamma_N+1}\right\}\prod_{i=1}^{N}dP(\tilde{g}_i)
\end{align}
where the integral in right hand side of (\ref{eq4.3}) is of constant magnitude. So we can rearrange both sides of (\ref{eq4.2}), and cancel the exponential term $\exp\{N(\theta v-\frac{1}{2N}\sum_{i=1}^N\log(1-2\theta\lambda_i(B_N)+2\theta v))\}$. (\ref{eq4.3}) is equivalent to prove
\begin{align*}
\int_{I_\epsilon^{c}\cap \mathcal{D}}\exp\left\{-\theta N\frac{\gamma_N(\hat{\gamma}_N-v\gamma_N)}{\gamma_N+1}\right\}\prod_{i=1}^{N}dP(\tilde{g}_i)\leq \eta(\kappa_1,\kappa_2,\epsilon, N) \int_{\mathcal{D}}\exp\left\{-\theta N\frac{\gamma_N(\hat{\gamma}_N-v\gamma_N)}{\gamma_N+1}\right\}\prod_{i=1}^{N}dP(\tilde{g}_i).
\end{align*}
Take a close look of the integral region of the left hand side of the above expression,
\begin{align*}
I_\epsilon^{c}\cap \mathcal{D}=&\left(\cup_{i=1}^{N}\{|g_i|\geq \sqrt{1+\gamma_N}N^{\epsilon}\}\right)\cap \mathcal{D}
                        \subset\left(\cup_{i=1}^{N}\{|g_i|\geq\frac{1}{2}N^{\epsilon}\}\right)\cap \mathcal{D}
\end{align*}
With these and expression (\ref{eq5}) in section \ref{sec3}, we have
\begin{align*}
&\int_{I_\epsilon^{c}\cap \mathcal{D}}\exp\left\{-\theta N\frac{\gamma_N(\hat{\gamma}_N-v\gamma_N)}{\gamma_N+1}\right\}\prod_{i=1}^{N}dP(\tilde{g}_i)
\leq CN \int_{\{|g_1|\geq \frac{1}{2}N^{2\epsilon}\}\cap \mathcal{D}}\exp\left\{-\theta N\gamma_N(\hat{\gamma}_N-v\gamma_N)\right\}\prod_{i=1}^{N}dP(\tilde{g}_i)\\
&\leq \frac{CN}{2\pi}\int_{I_1\times I_2} e^{ibx_1\sqrt{N}+bx_2\sqrt{N}}\left\{ \int_{|g_1|\geq \frac{1}{2}N^{2\epsilon}} \prod_{i=1}^{i=N}\frac{1}{\sqrt{2\pi}}e^{-\frac{\tilde{g}_i^2}{2}\left(1+\frac{2b(i(1-v+\lambda_i)x_1+(1+v-\lambda_i)x_2)}{\sqrt{N}(1-2\theta\lambda_i+2\theta v)}\right)}d\tilde{g}_i\right\} \prod_{i=1}^{2}e^{-\frac{x_i^2}{2}}dx_i \\
&\leq CNe^{-c'N^{2\epsilon}}\int_{I_1\times I_2} e^{ibx_1\sqrt{N}+bx_2\sqrt{N}}\left\{ \int \prod_{i=1}^{i=N}\frac{1}{\sqrt{2\pi}}e^{-\frac{\tilde{g}_i^2}{2}\left(1+\frac{2b(i(1-v+\lambda_i)x_1+(1+v-\lambda_i)x_2)}{\sqrt{N}(1-2\theta\lambda_i+2\theta v)}\right)}d\tilde{g}_i\right\} \prod_{i=1}^{2}e^{-\frac{x_i^2}{2}}dx_i\\
&\leq CNe^{-c'N^{2\epsilon}}\int_{\mathcal{D}}\exp\left\{-\theta N\frac{\gamma_N(\hat{\gamma}_N-v\gamma_N)}{\gamma_N+1}\right\}\prod_{i=1}^{N}dP(\tilde{g}_i),
\end{align*}
where the term $e^{-N^{2\epsilon}}$ comes from the integral of $\tilde{g}_1$ on the region $|\tilde{g}_1|\geq \frac{1}{2}N^{\epsilon}$, and $C$ is some constant, and may be different in different line. Therefore we can take $\eta(\kappa_1,\kappa_2, \epsilon, N)$ to be $CN e^{-c'N^{2\epsilon}}$. This finishes the proof.
\end{proof}

Since the integral outside the region $I_\epsilon$ is negligible, we only need to consider $\mathbb{E}_{X_N}(e^{\theta N\langle e, X_Ne \rangle})$ for those $e\in I_\epsilon$.
\begin{align*}
\mathbb{E}_{X_N}\left[e^{\theta N \langle e, X_N e\rangle}\right]
=&\prod_{1\leq i<j\leq N}\mathbb{E}[e^{2\theta \sqrt{N} e_ie_j(\sqrt{N}X_{ij})}]
\prod_{1\leq i\leq N} \mathbb{E}[e^{\theta \sqrt{N}e_i^2(\sqrt{N}X_{ii})}]\\
\leq & \prod_{1\leq i<j\leq N}e^{2\theta^2 N e_i^2e_j^2+c|4\theta^3N^{\frac{3}{2}}e_i^3e_j^3|}
\prod_{1\leq i\leq N}e^{\frac{1}{2}\theta^2Ne_i^4+c|\theta^3N^{\frac{3}{2}}e_i^6|}\\
\leq & e^{(N\theta^2+2c|\theta|^3N^{\frac{1}{2}+2\epsilon})\sum_{i,j}e_i^2e_j^2+\frac{\theta^2}{2}N^{2\epsilon}}\\
=&e^{N\theta^2+2c|\theta|^3N^{\frac{1}{2}+2\epsilon}},
\end{align*}
where the first inequality comes from Lemma \ref{lemp}. Similarly we can get the lower bound,
\begin{align*}
e^{N\theta^2-2c|\theta|^3N^{\frac{1}{2}+2\epsilon}}
\leq \mathbb{E}_{X_N}\left[e^{\theta N \langle e, X_N e\rangle}\right]\leq e^{N\theta^2+2c|\theta|^3N^{\frac{1}{2}+2\epsilon}}.
\end{align*}
Combining this and Proposition (\ref{pro7}), we have
\begin{align*}
\mathbb{E}_{X_N}[I_N(\theta, A_N+X_N)]
=&(1+o(1))\int_{I_\epsilon} e^{\theta N\langle e, A_N e\rangle} \mathbb{E}_{X_N}[e^{\theta N \langle e, X_N e\rangle}] d\mathbb{P}(e)\\
=&e^{N\theta^2+O(N^{\frac{1}{2}+2\epsilon})}\int e^{\theta N\langle e, A_N e\rangle}d\mathbb{P}(e).
\end{align*}
Therefore, taking $\log$ on both sides and divide by $N$, the term $\exp\{O(N^{\frac{1}{2}+2\epsilon})\}$ vanishes. This finishes the proof of Proposition \ref{pro1}.
\end{proof}

The proof of Proposition \ref{pro2} follows exactly the same strategy as in section (6.2) of \cite{GuMa}. Since the entries of $X_N$ satisfy log-Sobolev inequality with coefficient $\frac{c}{N}$, and $\frac{1}{N}\log I(\theta, A_N+\cdot)$ is a Lipschitz function, with Lipschitz constant $2|\theta|$, we have the following concentration of measure inequality
\begin{align}
\label{con1}\mathbb{P}\left(\left|\frac{1}{N}\log I(\theta, A_N+X_N)-\mathbb{E}_{X_N}\left[\frac{1}{N}\log I(\theta, A_N+X_N)\right]\right|\geq \delta\right)\leq 2e^{-\frac{N\delta^2}{4|\theta|c}}.
\end{align}

Similarly to the Lemma 25 of \cite{GuMa}, we have the following lemma, which provides us a lower bound on $\mathbb{P}(I(\theta, A_N+X_N)\geq \frac{1}{2} \mathbb{E}(I(\theta, A_N+X_N)))$. And the Proposition \ref{pro2} follows from this lower bound and the concentration inequality (\ref{con1}).
\begin{lemma}
Let $(A_N, X_N)_{N\in \mathbb{N}}$ be a sequence of deterministic real Hermitian matrices and Wigner matrices as in Theorem (\ref{thm1}). Then for $\theta$ small enough and any $\epsilon>0$, there exist finite constants $C$, $c$ which depends on $\epsilon$, $\theta$, $A_N$ and $B_N$, such that
\begin{align*}
\frac{\mathbb{E}_{X_N}\left[\int e^{\theta N\langle e, (A_N+X_N)e\rangle}d\mathbb{P}(e)\right]^2}
{\left(\mathbb{E}_{X_N}\left[\int e^{\theta N\langle e, (A_N+X_N)e\rangle}d\mathbb{P}(e)\right]\right)^2}\leq C e^{c N^{\frac{1}{2}+2\epsilon}}
\end{align*}
\end{lemma}
\begin{proof}
The proof is analogue to the proof of Lemma 25 of \cite{GuMa}, so we omit it here.
\end{proof}

\section*{Acknowledgement}
This research was conducted at the Undergraduate Research Opportunities Program of the MIT Mathematics Department, under the direction of Prof. Alice Guionnet. I would like to express to her my warmest thanks both for introducing me to this problem and for her dedicated guidance throughout the research process.

\begin{appendix}
\setcounter{lemma}{0}
\renewcommand{\thelemma}{\Alph{section}\arabic{lemma}}

\section{Gaussian Integral Formula}
\begin{lemma}\label{CI}
Given $n$ by $n$ symmetric matrix $K$, whose real part is positive definite, then we have the following change of variable formula for Gaussian type integral,
\begin{align*}
\int_{\mathbb{R}^n} F(x) e^{-\frac{1}{2}\langle x, K x\rangle}dx=
\det(A)\int_{\mathbb{R}^n} F(Ay+b) e^{-\frac{1}{2}\langle Ay+b, K (Ay+b)\rangle}dy
\end{align*}
where $F$ is a polynomial in terms of $\{x_i\}_{i=1}^{n}$, $A=diag\{a_1,a_2\cdots,a_n\}$ and $\Re(a_i)>0$ for $i=1,2\cdots n$, and $b\in \mathbb{C}^n$.
\end{lemma}
\begin{proof}
This can be proved by reducing to one dimensional case.
\end{proof}
\begin{lemma}\label{GIF}
Given $n$ by $n$ symmetric matrix $K$, whose real part is positive definite, then we have the following integral formula for any polynomial $F$( or even power series) in the variables $\{x_i\}_{i=1}^{n}$,
\begin{align*}
\int_{\mathbb{R}^n} F(x) e^{-\frac{1}{2}\langle x, K x\rangle}dx=\frac{\sqrt{2\pi}^n}{\sqrt{\det(K)}}\left\{F(\frac{\partial}{\partial \xi})e^{\frac{1}{2}\langle\xi, K^{-1}\xi \rangle}\right\}\Big{|}_{\xi=0}.
\end{align*}
\end{lemma}
\begin{proof}
Consider the Laplacian transform,
\begin{align*}
\int_{\mathbb{R}^n} e^{\langle x, \xi \rangle} F(x)e^{-\frac{1}{2}\langle x, K x\rangle}dx = \int F(\frac{\partial}{\partial \xi})e^{\langle x, \xi \rangle} e^{-\frac{1}{2}\langle x, K x \rangle}dx.
\end{align*}
Since $e^{-\frac{1}{2}\langle x, Kx\rangle}$ is a Schwartz function, with decaying speed faster than $e^{\langle x, \xi\rangle}$. By dominated convergence theorem, we can interchange the integral and differential.
\begin{align*}
\int_{\mathbb{R}^n} e^{\langle x, \xi \rangle} F(x)e^{-\frac{1}{2}\langle x, K x\rangle}dx
=& F(\frac{\partial}{\partial \xi})\left\{\int_{\mathbb{R}^n} e^{\langle x, \xi \rangle} e^{-\frac{1}{2}\langle x, K x \rangle}dx\right\}\\
=& F(\frac{\partial}{\partial \xi})\left\{e^{
\frac{1}{2}\langle \xi, K^{-1}\xi\rangle}\int_{\mathbb{R}^n} e^{-\frac{1}{2}\langle (x-K^{-1}\xi), K (x-K^{-1}\xi) \rangle}dx\right\}.
\end{align*}
From Lemma \ref{CGI} and Lemma \ref{CI}, we get
\begin{align*}
\int_{\mathbb{R}^n} F(x)e^{-\frac{1}{2}\langle x, K x\rangle}dx
=&\int_{\mathbb{R}^n} e^{\langle x, \xi \rangle} F(x)e^{-\frac{1}{2}\langle x, K x\rangle}dx\Big{|}_{\xi=0}\\
=&\frac{\sqrt{2\pi}^{n}}{\sqrt{\det(K)}}\left\{F(\frac{\partial}{\partial \xi})e^{\frac{1}{2}\langle\xi, K^{-1}\xi \rangle}\right\}\Big{|}_{\xi=0}.
\end{align*}
\end{proof}

\section{Detailed Computation for $m_1$}
In this section we give the detailed computation for coefficient $m_1$. To compute $f_0$, $f_1$, $f_2$, take $l=0,1,2$ in (\ref{key}) and follow the process in Page $14,15$. It is not hard to derive,
\begin{align}
\notag f_0(\theta)=&\frac{1}{2\pi}\int \left(1+\frac{1}{N}p_0(x_1,x_2)\right)e^{-\frac{1}{2}\langle x, Kx \rangle}dx+O(N^{-2})\\
\label{ma1}=&\frac{1}{\sqrt{\det(K)}}(1+\frac{1}{N}p_0(\frac{\partial}{\partial \xi_1},\frac{\partial}{\partial\xi_2})e^{\frac{1}{2}{\langle \xi, K^{-1}\xi\rangle}})\Big{|}_{\xi=0}+O(N^{-2}),
\end{align}
where
\begin{align*}
p_0(x_1,x_2)=\frac{\theta^2}{2}\sum_{i=1}^{N} \frac{(i(1-v+\lambda_i)x_1+(1+v-\lambda_i)x_2)^4}{N(1-2\theta \lambda_i+2\theta v)^4}+
\frac{\theta^3}{9}\left\{\sum_{i=1}^{N} \frac{(i(1-v+\lambda_i)x_1+(1+v-\lambda_i)x_2)^3}{N(1-2\theta \lambda_i+2\theta v)^3}\right\}^2
\end{align*}
And for $f_1$ and $f_2$,
\begin{align}
\notag f_i(t)=&\frac{1}{N}\frac{1}{2\pi}\int p_i(x_1,x_2)e^{-\frac{1}{2}\langle x, \tilde{K}x \rangle}dx+O(N^{-2})\\
\label{ma2}=&\frac{1}{N}\frac{1}{\sqrt{\det( \tilde{K})}}p_i(\frac{\partial}{\partial \xi_1},\frac{\partial}{\partial \xi_2})e^{\frac{1}{2}\langle \xi, \tilde{K}^{-1}\xi\rangle}\Big{|}_{\xi=0}+O(N^{-2})
\end{align}
where
\begin{align*}
p_1(x_1,x_2)=& \frac{2t^2}{3}\left\{\sum_{i=1}^{N}\frac{(i(1-v+\lambda_i)x_1+(1+v-\lambda_i)x_2)}{N(1-2\theta\lambda_i+2\theta v)^{2}}\right\}\left\{\sum_{i=1}^{N} \frac{(i(1-v+\lambda_i)x_1+(1+v-\lambda_i)x_2)^3}{N(1-2\theta \lambda_i+2\theta v)^3}\right\}\\
+&2t\sum_{i=1}^{N}\frac{(i(1-v+\lambda_i)x_1+(1+v-\lambda_i)x_2)^{2}}{N(1-2\theta\lambda_i+2\theta v)^{3}}\\
p_2(x_1,x_2)=&\sum_{i=1}^{N}\frac{2}{N(1-2\theta \lambda_i+2\theta v)^2}+2t\left\{\sum_{i=1}^{N}\frac{(i(1-v+\lambda_i)x_1+(1+v-\lambda_i)x_2)}{N(1-2\theta\lambda_i+2\theta v)^{2}}\right\}^2.
\end{align*}
For the integral (\ref{ma1}) and (\ref{ma2}), it is merely symbolic computation, and we can easily do it by some mathematic software, like Mathematica, then get explicit formula for $f_0$, $f_1$ and $f_2$.
\begin{align*}
f_0(\theta)=&\frac{1}{\sqrt{A_2}}
+\frac{1}{N}\frac{1}{\sqrt{A_2}}(\frac{7}{6}-\frac{3}{A_2}
+\frac{2A_3}{A_2^2}-\frac{5}{3}\frac{A_3^2}{A_2^3}+\frac{3}{2}\frac{A_4}{A_2^2})+O(N^{-2}),\\
f_1(t)=&\frac{1}{N}\left\{\frac{ 2t(t-\theta)(2tA_2^2-A_3(t-2\theta)-A_2(t+2\theta))}{\theta^3(\det(\tilde{K}))^{\frac{5}{2}}}\right\}+O(N^{-2})\\
f_2(t)=&\frac{1}{N}\frac{ 2A_2}{(\det(\tilde{K}))^{\frac{3}{2}}}+O(N^{-2}).
\end{align*}
Actually in this way, we can obtain any higher expansion terms of the spherical integral.
\end{appendix}

\bibliographystyle{Abbrv}

\end{document}